\documentclass[a4paper,11pt,english]{article}
\usepackage[utf8]{inputenc}
\usepackage[T1]{fontenc}
\usepackage{babel}
\usepackage[babel=true]{csquotes}
\usepackage{amsfonts,dsfont}
\usepackage{fullpage}
\usepackage{units}
\usepackage{multirow}
\usepackage{amsmath,amsthm,amsfonts,amssymb,upgreek,eufrak}
\usepackage{graphicx}
\usepackage{color}
\usepackage[all]{xy}
\usepackage{hyperref}

\usepackage{url}

\usepackage{cite}

\usepackage{titlesec}

\titleformat{\subsection}
{\normalfont\fontsize{11pt}{13pt}\selectfont}% apparence commune au titre et au numéro
{\rm\thesubsection}% apparence du numéro
{1em}% espacement numéro/texte
{\it}% apparence du titre

\usepackage{titlesec}
\titleformat{\subsubsection}
{\normalfont\fontsize{11pt}{13pt}\selectfont}% apparence commune au titre et au numéro
{\rm\thesubsubsection}% apparence du numéro
{1em}% espacement numéro/texte
{\it}% apparence du titre

\newtheorem{thmm}{Theorem}[section]
\newtheorem{prop}{Proposition}[section]
\newtheorem{lem}{Lemma}[section]
\newtheorem{coro}{Corollary}[section]
\newtheorem{defin}{Definition}[section]
\newtheorem{step}{Step}
\newtheorem{step2}{Step}

\newcommand{\ve}{\varepsilon}
\newcommand{\dd}{{\rm d}}

\begin{document}

\title{\bf \Large Invariant distributions and scaling limits for some diffusions in time-varying random environments}

\author{\large \it Yoann Offret}

\date{}

\maketitle

\noindent
\rule{\linewidth}{.5pt}

\begin{center}
\scriptsize \sl Université de Neuchâtel, Institut de Mathématiques\\ 
Bâtiment UniMail, Rue Émile Argant 11\\
2000 Neuchâtel Suisse\\[5pt]
Ph.: +41 (0)32 718 28 07\\
Fax: +41 (0)32 718 28 01\\[5pt]
%Email: {\tt yoann.offret@uninee.ch}\\
\href{mailto:yoann.offret@uninee.ch}{\nolinkurl{yoann.offret@unine.ch}}\\
%Web P.: {\tt http://members.unine.ch/yoann.offret/}
\href{http://members.unine.ch/yoann.offret/}{http://members.unine.ch/yoann.offret/}
\end{center}

\noindent
\rule{\linewidth}{.5pt}

\vspace{2em}

%~ \rule{15cm}{.5pt}
%~ 
%~ 
%~ \begin{minipage}[t]{0.47\textwidth}
%~ \small Université de Neuchâtel\\
%~ Institut de Mathématiques\\
%~ Bâtiment UniMail \\
%~ Rue Émile Argant 11\\
%~ 2000 Neuchâtel Suisse\\
%~ Ph.: +41 (0)32 718 28 07\\
%~ Fax: +41 (0)32 718 28 01
%~ \end{minipage}
%~ \hfill
%~ \begin{minipage}[t]{0.685\textwidth}
%~ \small Email: \href{mailto:yoann.offret@uninee.ch}{\nolinkurl{yoann.offret@unine.ch}}\\
%~ Web P.: \href{http://members.unine.ch/yoann.offret/}{http://members.unine.ch/yoann.offret/}	
%~ \end{minipage}
%~ 
%~ \vspace{1em}
%~ 
%~ \rule{15cm}{.5pt}
%~ 
%~ 
%~ \vspace{1em}

%%%%%%%%%%%%%%%%%%%%%%%%%%%%%%%%%%%%%%%%%%%%%%%%%%%%%%%%%%%%%%%%%%%%%%%%%%%%%%%%%%%%%%%%%%%%%%%%%%%%%%%%%%%%%%%%%%%%%%%%%%%%%%%%%%%%%%%%%%%%%%%%%%%%%%%%%%%%%%%%%%%%%%%%%%%%%%%%%%%%%%%%%%%%%%%%%%%%%%%%%%%%%%%%%%%%%%%%%%%%%%%%%%%%%%%%%%%%%%%%%%%%%%%%%%%%%%%%%%%%%%%%%%%%%%%%%%%%%%%%%%%%%%%%%%%%%%%%%%%%%%%%%%%%%%%%%%%%%%%%%%%%%%%%%%%%%%%%

\noindent
{\small {\bf Abstract}\; We consider a family of one-dimensional diffusions, in dynamical Wiener mediums, which are random perturbations of the Ornstein-Uhlenbeck diffusion process. We prove quenched and annealed convergences in distribution and under weigh-ted total variation norms. We find two kind of stationary probability measures, which are either the standard normal distribution or a quasi-invariant measure, depending on the environment, and which is naturally connected to a random dynamical system. We apply these results to the study of a model of time-inhomogeneous Brox's diffusions, which generalizes the diffusion studied by Brox (1986) and those investigated by Gradinaru and Offret (2011). We point out two distinct diffusive behaviours and we give the speed of convergences in the quenched situations.}\\

\noindent
{\small {\bf Key words}\; Time-dependent random environment .  Time-inhomogeneous Brox's diffusion .  Random dynamical system .  
Foster-Lyapunov drift condition .  Fluctuating stationary distribution}\\


\vspace{2em}

\noindent
\rule{\linewidth}{.5pt}

\tableofcontents
{}
\vspace{2em}

\noindent
\rule{\linewidth}{.5pt}

%%%%%%%%%%%%%%%%%%%%%%%%%%%%%%%%%%%%%%%%%%%%%%%%%%%%%%%%%%%%%%%%%%%%%%%%%%%%%%%%%%%%%%%%%%%%%%%%%%%%%%%%%%%%%%%%%%%%%%%%%%%%%%%%%%%%%%%%%%%%%%%%%%%%%%%%%%%%%%%%%%%%%%%%%%%%%%%%%%%%%%%%%%%%%%%%%%%%%%%%%%%%%%%%%%%%%%%%%%%%%%%%%%%%%%%%%%%%%%%%%%%%%%%%%%%%%%%%%%%%%%%%%%%%%%%%%%%%%%%%%%%%%%%%%%%%%%%%%%%%%%%%%%%%%%%%%%%%%%%%%%%%%%%%%%%%%%%%%%%%%%%%%%%%%%%%%%%%%%%%%%%%%%%%%%%%%%%%%%%%%%%%%%%%%%%%%%%%%%%%%%%%%%%%%%%%%%%%%%%%%%%%%%%%%%%%%%%%%%%%%%%%%%%%%%%%%%%%%%%%%%%%%%%%%%%%%%%%%%%%%%%%%%%%%%%%%%%%%%%%%%%%%%%%%%%%%%%%%%%%%%%%%%%%%%%%%%%%%%%%%%%%%%%%%%%%%%%%%%%%%%%%%%%%%%%%%%%%%%%%%%%%%%%%%%%%%%%%%%%%%%%%%%%%%%%%%

\section{Introduction}
\label{sec:1}
\setcounter{equation}{0}

Random walks (RWs) in random environments (REs) and their continuous-time counterparts, the diffusions in random environment, pave the way for the study of a multitude of interesting cases, which have been tackled since the $70{\mbox{\textquoteright}}{\rm s}$ in a large section of the literature. 

Concerning the genesis of the theory, we allude to \cite{KKS,Sin}, as regards the discrete-time situation, and to \cite{Schumi,Brox,KTT}, as regards the continuous-time one. For more recent refinements and generalizations, we refer to \cite{mathieu1,hushi,HSY,shi,schm,ZeiSz,chelio,DGPS,CGZ,diel}
  and for a general review of the topic, we refer to \cite{Zei1}.

Here we investigate one-dimensional diffusions evolving in  dynamical  Wiener media, which have some common features with those studied in \cite{Brox,MY}. We give, under weighted total variation norms, quenched and annealed diffusive scaling limits, which may depend on the environment, and thus, which are not always normal distributions. We also give the speeds of convergence under the  quenched distributions. In addition, we bring out a phase transition phenomenon, which is the analogue in RE, to a particular situation considered  in \cite{MY}. 

RWs in dynamical REs have been widely and intensively considered in the past few years under several assumptions. Initially, space-time i.i.d. REs have been introduced and studied in \cite{Boldri,Boldri3,Seppa}. 
Further difficulties arise when the fluctuations of the REs are i.i.d. in space and Markovian in time, case addressed in \cite{Zei2,Dolgo}, and major one arise when we consider space-time mixing REs, case recently studied in \cite{Dolgo1,Bric,Avena}.
However, continuous-time diffusions in time-varying  random environment have been sparsely investigate. Nevertheless, we can mention \cite{LOY,KO,KO2,Rhodes} concerning the homogenization of diffusions in time-dependent random flows.

\subsection{The Wiener space}
Introduce the space
\begin{equation}\label{env}
\Theta:=\Big\{\theta\in {\rm C}(\mathbb R;\mathbb R) : \theta(0)=0\;\mbox{and}\;\lim_{|x|\to\infty}{x^{-2}\,\theta(x)}=0\Big\}
\end{equation}
endowed by the standard $\sigma$-field $\mathcal B$ generated by the Borel cylinder sets. It is classical that there exists a unique probability measure $\mathcal W$ on $(\Theta,\mathcal B)$ such that the processes $\{\theta(\pm\,x) : \theta\in\Theta, x\geq 0\}$ are two independent standard Brownian motions. The probability distribution $\mathcal W$ is called the Wiener measure. We denote by $\{S_\lambda : \lambda>0\}$ the scaling transformations on $\Theta$ defined by 
\begin{equation}\label{scaling}
S_\lambda\theta(\ast) := \frac{\theta(\lambda \ast)}{\sqrt \lambda}.
\end{equation}
Note that $\Theta$ is naturally endowed with a structure of separable Banach space, such that $\mathcal B$ coincides with the Borel $\sigma$-field $\mathcal B_\Theta$.

\subsection{Schumacher and Brox's results}

Brox makes sense in \cite{Brox} to solution of the informal diffusion equation
\begin{equation}\label{brox}
\dd X_t = \dd B_t-\frac{1}{2}\theta^\prime(X_t)\, \dd t,
\end{equation}
where $\theta\in\Theta$ and $B$ is a standard one-dimensional Brownian motion independent of the Brownian environment $(\Theta,\mathcal B,\mathcal W)$. Denoting by $\mathds P_\theta$ and $\widehat {\mathds P}$, respectively the quenched and annealed distributions (the expectation of $\mathds P_\theta$ under $\mathcal W$) of such solution, Schumacher and Brox show, independently in \cite{Schumi0,Schumi} and \cite{Brox}, that there exists a family of 
 measurable functions $\{b_h : h>0\}$ on $(\Theta,\mathcal B)$ such that the following convergence holds in probability
\begin{equation}\label{brox1}
\frac{X_t}{(\log t)^2}-b_1\big(S_{(\log t)^2}\theta\big) = \frac{X_t - b_{\log t}(\theta)}{(\log t)^2} \xrightarrow[t\to\infty]{\widehat {\mathds P}} 0.
\end{equation}
The Wiener measure being invariant under the scaling transformations, if we denote by $\hat b_1$ the distribution of $b_1$ under $\mathcal W$, the following annealed convergence holds in distribution
\begin{equation}\label{brox2}
\frac{X_t}{(\log t)^2}\xrightarrow[t\to\infty]{(\dd)} \hat b_1.
\end{equation}
The key to prove these results is to take full advantage of the representation of $X$ in terms of a one-dimensional Brownian motion changed in scale and time, and of the invariance of the Brownian motions $B$ and $\theta$ under the scaling transformations $S_{\lambda}$. The authors prove that the diffusion is localized in the valleys of the potential $\theta$, which are themselves characterized by $b_1$. 

\subsection{Phase transition in a $2$-stable deterministic environment}

Set $W(x):=|x|^{{1}/{2}}$ and consider, for any $\beta\in\mathbb R$, the particular time-inhomogeneous singular stochastic differential equation (SDE) studied in \cite{MY} and which is given by
\begin{equation}\label{my}
\dd Y_t = \dd B_t-\frac{1}{2}\frac{W^\prime(Y_t)}{t^\beta}\, \dd t.
\end{equation}
The authors show in \cite{MY} the existence of a pathwise unique strong solution and prove  diffusive and subdiffusive scaling limits in distribution,  depending on the position of $\beta$ with respect to $1/4$. More precisely, they prove that 
\begin{equation}\label{my1}
\frac{Y_t}{\sqrt t}\xrightarrow[t\to\infty]{(\dd)} 
\left\{\begin{array}{l}
\mathcal N(0,1),\quad\mbox{when}\;\beta>1/4,\\
  \\
k_c^{-1}{e^{-\big[\frac{x^2}{2}+W(x)\big]}\,\dd x}, \quad \mbox{when}\;\beta=1/4,
\end{array}\right.
\end{equation}
and 
\begin{equation}\label{my2}
\frac{Y_t}{t^{2\beta}}\xrightarrow[t\to\infty]{(\dd)} k_u^{-1}{e^{-W(x)}\,\dd x},\quad\mbox{when}\; \beta<1/4,
\end{equation}
$k_c$ and $k_u$ being two normalization positive constants. In fact, to obtain the convergences in  $(\ref{my1})$, they study the diffusion equation 
\begin{equation}\label{diff_my_intro}
\dd Z_t= \dd B_t-\frac{1}{2}\left[Z_t+e^{-rt}\,W^\prime(Z_t)\right]\, \dd t.
\end{equation}
This process is naturally related to equation (\ref{my}) by setting $r:=\beta-1/4$, via a well chosen scaling transformation taking full advantage of the scaling property of the Brownian motion $B$ and of the deterministic scaling property of the potential $W$. For more details , we refer to \cite{MY}. We can expect to obtain similar results by replacing $W$ in equation (\ref{my}) by a typical Brownian path $\theta\in\Theta$, a $2$-stable random process, and this is one of the main objects of this article.

\subsection{Overview of the article}

The paper is organized as follows: in Section \ref{sec:2}, we introduce a diffusion equation (\ref{diff_intro}) in a dynamical Wiener potential, which generalizes equation (\ref{diff_my_intro}). Then we state our main results and we give the general strategy of the proofs. In section \ref{sec:3}, we apply these results to  a model of time-inhomogeneous Brox's diffusions. This is a generalization of equation (\ref{my}) and (\ref{brox}) and we obtain similar asymptotic behaviours as in (\ref{my1}). Thereafter, in Section \ref{sec:4}, we introduce some linear perturbations of equation (\ref{diff_intro}). We show some properties, related to these ones, which are used in Sections \ref{sec:5} and \ref{sec:6} to prove existence, uniqueness and nonexplosion for the diffusion process (\ref{diff_intro}) (Theorem \ref{existence_intro}) and also to prove that this process is a strongly Feller diffusion satisfying the lower local Aronson estimate and a kind of cocycle property (Theorem \ref{markov_feller_intro}). In Section \ref{sec:7}, we prove some technical results in order to obtain the quenched and annealed convergences (Theorems \ref{invariant_intro} and \ref{CG}) in the two last Sections.

%%%%%%%%%%%%%%%%%%%%%%%%%%%%%%%%%%%%%%%%%%%%%%%%%%%%%%%%%%%%%%%%%%%%%%%%%%%%%%%%%%%%%%%%%%%%%%%%%%%%%%%%%%%%%%%%%%%%%%%%%%%%%%%%%%%%%%%%%%%%%%%%%%%%%%%%%%%%%%%%%%%%%%%%%%%%%%%%%%%%%%%%%%%%%%%%%%%%%%%%%%%%%

\section{Model and statement of results}
\label{sec:2}
\setcounter{equation}{0}

%%%%%%%%%%%%%%%%%%%%%%%%%%%%%%%%%%%%%%%%%%%%%%%%%%%%%%%%%%%%%%%%%%%%%%%%%%%%%%%%%%%%%%%%%%%%%%%%%%%%%%%%%%%%%%%%%%%%%%%%%%%%%%%%%%%%%%%%%%%%%%%%%%%%%%%%%%%%%%%%%%%%%%%%%%%%%%%%%%%%%%%%%%%%%%%%%%%%%%%%%%%%%

\subsection{Diffusions in a fluctuating Ornstein-Uhlenbeck potential}

In the present paper, we study Brownian motions dynamics, in time-dependent Wiener media, given by the underlying dynamical random environment 
\begin{equation}\label{scaling_intro}
\left\{T_t\theta(x) := S_{e^{{t}/{2}}}\theta(x)=e^{-{t}/{4}}\,\theta(e^{{t}/{2}}x) : \theta\in\Theta,\,t,x\in\mathbb R\right\}.
\end{equation} 
The family $\{T_t : t\in \mathbb R \}$ is a one-parameter group of transformations leaving invariant $\mathcal W$ and such that, under this probability measure, $\{T_{t}\theta(x) : t\in\mathbb R\}$ is a stationary Ornstein-Uhlenbeck process having $\mathcal N(0,x)$ as stationary distribution. Moreover, the dynamical system $(\Theta,\mathcal B,\mathcal W, (T_t)_{t\in \mathbb R})$ is ergodic (see Proposition \ref{ergodicity}). 

We consider, for any $r\in\mathbb R$, the diffusion process $Z$, solution of the informal SDE driven by a standard Brownian motion $B$, independent of $(\Theta,\mathcal B,\mathcal W)$,
\begin{equation}\label{diff_intro}
\dd Z_t = \dd B_t-\frac{1}{2}\partial_x V_{\theta}(t,Z_t)\, \dd t,\quad  Z_s  =  z\in\mathbb R,\;t\geq s\geq 0,\;\theta\in\Theta,
\end{equation}
with
\begin{equation}\label{potential_intro}
V_{\theta}(t,x):=\frac{x^{2}}{2}+ e^{-r t}\,T_{t}\theta(x).
\end{equation}
Note that when $\theta$ is equal to $W$, defined in (\ref{my}), $T_t\theta$ in (\ref{potential_intro}) is simply equal to $\theta$ and equation (\ref{diff_intro}) is nothing but equation (\ref{diff_my_intro}). The diffusion process $Z$ can be seen as a Brownian motion immersed in the random time-varying potential $\{V_{\theta}(t,\cdot) : t\in\mathbb R\}$, as well as an Ornstein-Uhlenbeck diffusion process, whose potential is perturbed by the  dynamical Wiener medium $\{e^{-r t}\,T_t\theta : t\in\mathbb R\}$. 
Moreover, one can see $Z$ as a distorted Brownian motion, whose drift is a Gaussian field $\{\Gamma(t,x) : t,x\in\mathbb R\}$ having mean function $m_{\Gamma}$ and covariance function $C_{\Gamma}$ (here a Dirac measure) given by
\begin{equation*}
m_{\Gamma}(t,x)=-\frac{x}{2}\quad\mbox{and}\quad C_{\Gamma}(t,x;s,z) = \frac{1}{4}{e^{-\left[r(t+s)+\frac{|t-s|}{4}\right]}}\,\delta(e^{t/2}x-e^{s/2}z).
\end{equation*}

We need to give a correct sense to solution of equation (\ref{diff_intro}). Formally, we can see $Z$ as the diffusion process, whose conditional infinitesimal generator, given $\theta\in\Theta$, is
\begin{equation}\label{ig_intro}
L_{\theta} := L_{\theta,t} +\frac{\partial}{\partial t} := \left[\frac{1}{2} e^{V_{\theta}(t,x)} \frac{\partial}{\partial x} \left(e^{-V_{\theta}(t,x)} \frac{\partial}{\partial x}\right)\right]+\frac{\partial}{\partial t}.
\end{equation}
The domain and the socalled generalized domain of $L_\theta$ are defined by
\begin{multline}\label{domains_intro}
D(L_\theta):=\left\{F\in{\rm C}^{1} : e^{-V_{\theta}}\partial_x F \in {\rm C}^{1}\right\}\quad\mbox{and}\\
\overline D(L_\theta):=\left\{F\in {\rm W}^{1,\infty}_{\rm loc} : e^{-V_{\theta}}\partial_x F \in  {\rm W}^{1,\infty}_{\rm loc}\right\}
\end{multline}
where ${\rm C}^{1}$ and ${\rm W}^{1,\infty}_{\rm loc}$ denote the space of real continuous functions $F(t,x)$ on $[s,\infty)\times \mathbb R$ such that the partial derivatives $\partial_t F$ and $\partial_x F$ (in the sense of distributions) exist and are respectively continuous functions and locally bounded functions.  

This kind of diffusion operators, with distributional drift, have been already study in \cite{RussoTrutnau,FlanRussoWolf}  in the case where the coefficients of the SDE do not depend on time. Rigorously speaking, a weak solution to equation (\ref{diff_intro}) is a solution to the martingale problem related to $(L_\theta,D(L_\theta))$.

\begin{defin}\label{solution_intro}
A continuous stochastic process $\{Z_t : t\geq s\}$ defined on a given filtered probability space is said to be a weak solution to equation (\ref{diff_intro}) if $Z_s=z$ and if there exists an increasing sequence of stopping times $\{\tau_n : n\geq 0\}$ such that, for all $n\geq 0$ and $F\in D(L_{\theta})$,
\begin{equation}\label{MP_Diff}
F(t\wedge \tau_{n} ,Z_{t\wedge \tau_n}) -\int_s^{t\wedge\tau_n} L_\theta F(u,Z_u)\,\dd u,\quad t\geq s,
\end{equation}
is a local martingale, with
\begin{equation}
\tau_{e}:= \sup_{n\geq 0}\inf\{t\geq s : |Z_t|\geq n\}=\sup_{n\geq 0}\tau_{n}.
\end{equation}
A weak solution is global when the explosion time satisfies $\tau_e=\infty$ a.s. and we said that the weak solution is unique if all the weak solutions have the same distribution.
\end{defin}

We are now able to state our first result.

\begin{thmm}\label{existence_intro} For any $r\in\mathbb R$, $\theta\in\Theta$, $s\geq 0$ and $z\in\mathbb R$, there exists a unique global weak solution $Z$ to equation (\ref{diff_intro}). Moreover, there exists a standard Brownian motion $B$ such that, for all $F\in \overline D(L_\theta)$,
\begin{equation}\label{ito_diff_intro}
F(t ,Z_t) = F(s,z)+\int_s^{t} L_\theta F(u,Z_u)\,\dd u +\int_s^{t}\partial_x F(u,Z_u)\,\dd B_u,\quad t\geq s.\\
\end{equation}
\end{thmm}

Since the one-dimensional equation (\ref{diff_intro}) is not time-homogeneous, there are not simple conditions which characterize the nonexplosion as in \cite{Brox,RussoTrutnau,FlanRussoWolf}. Therefore, the main difficulty is to construct Lyapunov functions. To this end, we consider some linear perturbations of equation (\ref{diff_intro}), given in (\ref{diff_intro_g}), for which we are able, when the potential (\ref{potential_intro_g}) is sufficiently confining, to construct suitable Lyapunov functions (see Proposition \ref{nonexplosion}). Then we prove (see Theorem \ref{nonexplosion_f}) nonexplosion, existence and uniqueness (in a more general setting) by using the Girsanov transformation and by considering the SDE (\ref{sde}). This equation is connected to equation (\ref{diff_intro_g}), when the associated potential is attractive, via the pseudo-scale function $S_\theta$ defined in (\ref{pseudoscale}) (see Proposition \ref{equivalent_mp}). This method is a generalization in the time-inhomogeneous setting of that employed in \cite{Brox,FlanRussoWolf,RussoTrutnau} and which uses the effective scale function.

%%%%%%%%%%%%%%%%%%%%%%%%%%%%%%%%%%%%%%%%%%%%%%%%%%%%%%%%%%%%%%%%%%%%%%%%%%%%%%%%%%%%%%%%%%%%%%%%%%%%%%%%%%%%%%%%%%%%%%%%%%%%%%%%%%%%%%%%%%%%%%%%%%%%%%%%%%%%%%%%%%%%%%%%%%%%%%%%%%%%%%%%%%%%%%%%%%%%%%%%%%%%%

\subsection{Strong Feller property, cocycle property and lower local Aronson estimate}

In the following, we denote by $\mathds P_{s,z}(\theta)$ the distribution of the weak solution to equation (\ref{diff_intro}), called the quenched distribution, which existence is stated in Theorem \ref{existence_intro}. We introduce the canonical process $\{X_t: t\geq 0\}$ on the space of continuous functions from $[0,\infty)$ to $\mathbb R$, endowed with its standard Borel $\sigma$-field $\mathcal F$, and we denote by $P_\theta(s,z;t,\dd x)$ and $P_{s,t}(\theta)$, the probability transition kernel and  the associated Markov kernel defined, for all measurable nonnegative function $F$ on $\mathbb R$ by
\begin{equation}\label{semigroup_transition_intro}
{P}_{{s,t}}(\theta)F(z):=\mathds E_{s,z}(\theta)\left[F(X_t)\right] = \int_{\mathbb R} F(x) P_\theta(s,z;t,\dd x).
\end{equation}

\begin{thmm}\label{markov_feller_intro} For any $r\in\mathbb R$ and all $\theta\in\Theta$, the family $\{\mathds P_{s,z}(\theta) : s\geq 0,\, z\in\mathbb R\}$ is strongly Feller continuous. Moreover, the associated time-inhomogeneous semigroups $\{P_{s,t}(\theta): t\geq s\geq 0,\,\theta\in\Theta\}$ satisfy
\begin{equation}\label{presque_cocycle_intro}
{P}_{s,s+t}(\theta)={P}_{0,t}(e^{-r s}\,T_s\theta)
\quad\mbox{and}\quad P_{0,s+t}(\theta)=P_{0,s}(\theta)P_
{0,t}(e^{-r s}\,T_{s}\theta).
\end{equation}
Besides, $P_\theta(s,z;t,\dd x)$ admits a density $p_\theta(s,z;t,x)$, which is measurable with respect to  $(\theta,s,t,z,x)$ on  $\Theta\times \{t>s\geq 0\}\times \mathbb R^{2}$, and which satisfies the lower local Aronson estimate: for all $\theta\in \Theta$, $T>0$ and compact set $C\subset \mathbb R$, there exists $M>0$ such that, for all $0\leq s<t\leq T$ and $z,x \in C$,
\begin{equation}\label{lowerlocalaronson_intro}
p_{\theta}(s,z;t,x)\geq \frac{1}{\sqrt{M(t-s)}}\exp{\left(-M \frac{|z-x|^{2}}{t-s}\right)}. 
\end{equation}
\end{thmm}

The idea is to study the more general equivalent SDE (\ref{sde}) and to prove, by using standard technics, the analogous theorem for this diffusion (see Theorem \ref{markov_feller_sde}).

Besides, the transition density being measurable with respect to $\theta$, we can define the annealed distribution $\widehat{\mathds P}_{s,z}$ and the associated Markov kernel $\widehat P_{s,t}$ as 
\begin{equation*}
 \widehat{\mathds P}_{s,z}:=\mathds E_{\scriptscriptstyle\mathcal W}[\mathds P_{s,z}]:=\int_\Theta \mathds P_{s,z}(\theta)\,\mathcal W(\dd \theta)\;\;\mbox{and}\;\;
 {\widehat P}_{s,t}:=\mathds E_{\scriptscriptstyle\mathcal W}[P_{s,t}]:=\int_\Theta P_{s,t}(\theta)\,\mathcal W(\dd \theta).
\end{equation*}
We point out that 
$X$ is not a Markov process under $\widehat{\mathds P}_{s,z}$. Moreover, in the light of (\ref{presque_cocycle_intro}),
we can assume without loss of generality that $s=0$ in (\ref{diff_intro}) and we set
\begin{multline*}
\mathds P_{z}(\theta):=\mathds P_{0,z}(\theta),\quad P_\theta(z;t,\dd x)=P_\theta(0,z;t,\dd x),\quad p_\theta(z;t,x)=p_\theta(0,z;t,x),\\
P_{t}(\theta):=P_{0,t}(\theta),\quad \mbox{and}\quad  {\widehat P}_{t}:=\widehat P_{0,t}.
\end{multline*}
Furthermore, we can see that the case $r=0$ is of particular interest since the relation (\ref{presque_cocycle_intro}) can be written in this situation  
\begin{equation}\label{cocycle_intro}
{P}_{s,s+t}(\theta)={P}_{t}(T_s\theta)\quad\mbox{and}\quad P_{s+t}(\theta)=P_{s}(\theta)P_{t}(T_{s}\theta).
\end{equation}
Roughly speaking, the equation (\ref{diff_intro}) is time-homogeneous in distribution since from the scaling property $\mathcal W$ is $(T_t)$-invariant. Relation (\ref{cocycle_intro}) is called the cocycle property and it induces (see \cite{Arnold} for a definition) a random dynamical system (RDS) over $(\Theta,\mathcal B,\mathcal W, (T_t))$ on the set $\mathcal M$ of signed measures on $\mathbb R$, by setting, for all $\nu\in\mathcal M$,
\begin{equation*}
\nu P_t(\theta)(\dd x):=  \int_{\mathbb R} P_\theta(z;t,\dd x)\,\nu(\dd z)=\left(\int_{\mathbb R}p_\theta(z;t,x)\,\nu(\dd z)\right)\dd x.
\end{equation*}
Note that the subset of probability measures $\mathcal M_{1}\subset \mathcal M$ is invariant under this RDS.

%%%%%%%%%%%%%%%%%%%%%%%%%%%%%%%%%%%%%%%%%%%%%%%%%%%%%%%%%%%%%%%%%%%%%%%%%%%%%%%%%%%%%%%%%%%%%%%%%%%%%%%%%%%%%%%%%%%%%%%%%%%%%%%%%%%%%%%%%%%%%%%%%%%%%%%%%%%%%%%%%%%%%%%%%%%%%%%%%%%%%%%%%%%%%%%%%%%%%%%%%%%%%%%%%%%%%%%%%%%%%%%%%%%%%%%%%%%%%%%%%%%%%%%%%%%%%%%%%%%%%%%%%%%%%%%%%%%%%%%%%%%%%%%%%%%%%%%%%%%%%%%%%%%%%%%%%%%%%%%%%%%%%%%%%%%%%%%%%%%%%%%%%%%%%%%%%%%%%%%%%%%%%%%%%%%%%%%%%%%%%%%%%%%%%%%%%%%%%%%%%%%%%%%%%%%%%%%%%%%%%%%%%%%%%%%%%%%%%%%%%%%%%%%%%%%%%%%%%%%%%%%%%%%%%%%%%%%%%%%%%%%%%%%%%%%%%%%%%%%%%%%%%%%%%

\subsection{Quasi-invariant and stationary probability measures}

To state our next important results, we need to introduce some additional notations.  We said that $\mu$ is a random probability measure on $\mathbb R$, over $(\Theta,\mathcal B,\mathcal W)$, if $\mu_{\theta}\in\mathcal M_{1}$ for $\mathcal W$-almost all $\theta$, and if $\theta\longmapsto \mu_\theta(A)$ is measurable for all Borel set $A$. For such random probability measure $\mu$, we introduce the probability measure $\hat \mu$ defined by 
\begin{equation*} 
\hat \mu:= \mathds E_{\scriptscriptstyle\mathcal W}[\mu]:=\int_{\Theta} \mu_{\theta}\,\mathcal W(\dd\theta).
\end{equation*}
Let $\alpha\in\mathbb R$ and $U_\alpha$, $V_\alpha$ be the functions on $\mathbb R$ defined by
\begin{equation}\label{lyapunovfunctions_intro} 
U_\alpha(x):=\exp\left(\alpha\frac{x^2}{2}\right)\quad\mbox{and}\quad V_\alpha(x):=\exp(|x|^\alpha).
\end{equation}
The $F$-total variation norm, $F\in\{U_\alpha,V_{\alpha}\}$, of a signed measures $\nu$, is defined by
\begin{equation*}\label{totalvariation_intro}
\|\nu\|_{F} := \sup\left\{|\nu(f)| : |f|\leq F,\;\mbox{$f$ bounded and measurable}\right\}.
\end{equation*}
Note that if $\nu\in\mathcal M_1$ then $\|\nu\|_{F}=\nu(F)$. In addition, we set 
\begin{equation*}
\mathcal M_{F}:=\{\nu \in\mathcal M : \|\nu\|_{F}<\infty\}
\quad\mbox{and}\quad
\mathcal M_{1,F}=\mathcal M_1\cap \mathcal M_{F}.
\end{equation*}

\begin{thmm}\label{invariant_intro} Assume that $r=0$. There exists a random probability measure $\mu$ on $\mathbb R$ over $(\Theta,\mathcal B,\mathcal W)$, unique up to a $\mathcal W$-null set, such that, for all $t\geq 0$,
\begin{equation}\label{invariance_intro}
\mu_\theta P_t(\theta) =\mu_{T_t\theta}\quad\mathcal W\mbox{-a.s.}
\end{equation}
Moreover, for all $\alpha\in(0,1)$, the quasi-invariant measure satisfies
\begin{equation}\label{tail}
\mu_\theta \in\mathcal M_{1,U_{\alpha}}\quad\mathcal W\mbox{-a.s.}\quad\mbox{and}\quad \hat\mu \in\mathcal M_{1,V_{\alpha}}.
\end{equation}
Furthermore, there exists $\lambda>0$ such that, for all $\nu \in\mathcal M_{1,U_{\alpha}}$ and $\hat \nu\in\mathcal M_{1,V_{\alpha}}$,
\begin{equation}\label{weakergodicity_intro_1}
\limsup_{t\to\infty}\frac{\log(\left\|\nu P_{t}(\theta) - \mu_{T_t\theta}\right\|_{U_\alpha})}{t}\leq -\lambda\quad \mathcal W\mbox{-a.s.} 
\end{equation}
and 
\begin{equation}\label{weakergodicity_intro_2}
\lim_{t\to\infty}{\|\hat\nu \widehat P_{t} - \hat\mu \|_{V_\alpha}}=0.
\end{equation}
\end{thmm}

Linear RDSs have been studied in an extensive body of the literature. The dynamics (in particular the Lyapunov exponents) in the case where the discrete-time linear RDS acts on a finite dimensional space (the case of infinite products of random matrices) have been well understood for a long time, for instance in \cite{Osel,Guivarch},
whereas the situation where the general linear RDS acts on a separable Banach space has been newly studied in \cite{Zeng}.

Our goal in Theorem \ref{invariant_intro} is to obtain a quasi-invariant probability measure for the random Markov kernels $P_t(\theta)$ and to give convergence results in the separable Banach spaces $\mathcal M_{U_\alpha}$ (exponential convergence) and $\mathcal M_{V_\alpha}$. We need a kind of random Perron-Frobenius theorem, which has been, for example, obtained in \cite{ArnMatDem} for infinite products of nonnegative matrices, and more recently in \cite{kifer} for infinite products of stationary Markov kernels over a compact set. 

However, the Markov operators that we consider act on the infinite dimensional space $\mathcal M$ and are defined over the noncompact set $\mathbb R$. To overcome this problem, we need to see that $U_\alpha$ and $V_\alpha$ are  Foster-Lyapunov functions (see Propositions \ref{lyapunovUalpha} and \ref{lyapunovValpha}). More precisely, we show that Lyapunov exponents can be chosen independently of the environment $\theta$, while keeping a control on the expectation of the $U_\alpha$-norm and the $V_\alpha$-norm. The classical method to construct Foster-Lyapunov functions for Markov kernels is to construct Lyapunov functions for the infinitesimal generators (see Lemma \ref{step3} and \ref{step3bis}). Nonetheless, we stress that neither $U_\alpha$ nor $V_\alpha$ belong to the generalized domain $\overline D(L_\theta)$ and we need to approximate uniformly these functions by functions of this domain, while keeping a control on the expectation under the Wiener measure. This is possible by using the Hölder continuity of Brownian paths (see Proposition \ref{uniform_random_approximation}). 

Then, we use the explicit bound on convergence of time-inhomogeneous Markov chains (see Proposition \ref{A}), obtained from \cite{DoucMoulines}, via coupling constructions, Foster-Lyapunov conditions and the cocycle property, together with  the ergodicity of the underlying dynamical system $(\Theta,\mathcal B,\mathcal W,(T_t)_{t\in\mathbb R})$. We point out that the Aronson estimate (\ref{lowerlocalaronson_intro}) is necessary to the coupling constructions.

Furthermore, let us denote by $\{U_t : t\geq 0\}$ the canonical process on the space $\Xi$ of continuous functions from $[0,\infty)$ to $\Theta$, endowed with its standard Borel $\sigma$-field $\mathcal G$, and introduce the Markov kernels $\Pi_{\theta,z}$ on $(\Xi\times \Omega,\mathcal G\otimes\mathcal F)$, and the probability  measure $\overline \mu$ on $(\Theta\times \mathbb R,\mathcal B\otimes\mathcal B(\mathbb R))$, defined by the product and disintegration formula 
\begin{equation*}
\Pi_{\theta,z}:=\delta_{\{T_t\theta : t\geq 0\}}\otimes \mathds P_z(\theta)\quad\mbox{and}\quad \overline \mu(\dd\omega,\dd x):=\mathcal W(\dd\omega)\mu_\omega(\dd x).
\end{equation*}
Then we can see that $\{(U_t,X_t) : t\geq 0\}$ is a time-homogeneous Markov process under $\Pi_{\theta,z}$ such that $\overline \mu$ is an invariant initial distribution. This process is called the skew-product Markov process (see \cite{Cog1,Orey}
 for the discrete-time situation). By applying standard results on general time-homogeneous Markov processes (see for instance \cite{MeTw}) we deduce that for all $F\in{\rm L}^1(\Theta\times \mathbb R,\overline \mu)$, $z\in\mathbb R$ and $\mathcal W$ almost all  $\theta\in\Theta$, 
\begin{equation*}
\lim_{t\to\infty}\frac{1}{t}\int_0^t F(U_\tau,X_\tau)\, \dd \tau = \int_{\Theta\times \mathbb R} F(\omega,x) \,\overline \mu(\dd \omega,\dd x),\quad \Pi_{\theta,z}\mbox{-a.s.}
 \end{equation*} 
Note that equation (\ref{tail}) provides some information on the tails of $\mu_\theta$ and $\hat\mu$.{}

 %~ Especially, one can see that for all $\alpha\in(0,1)$, there exists a random variable $C : \Theta \longrightarrow (0,\infty)$ and $\widehat C>0$ such that, for all $R\geq 0$,
%~ \begin{equation}
%~ \mu_\theta((R,\infty))\leq C_\theta\,\exp{\left(-\alpha\frac{R^2}{2}\right)}\quad\mathcal W\mbox{-a.s.}\quad\mbox{and}\quad \hat \mu((R,\infty))\leq \hat C \exp\left(-R^\alpha\right).
%~ \end{equation}

\begin{thmm}\label{CG} Assume that $r>0$. For any $z\in\mathbb R$ and for $\mathcal W$-almost all $\theta\in\Theta$, the following convergence holds under the quenched distribution $\mathds P_z(\theta)$,
\begin{equation}\label{y3}
\lim_{t\to\infty} X_t \overset{(\dd)}{=} \mathcal N(0,1).	
\end{equation}
\end{thmm}

Here the space-time mixing environment is, contrary to  Theorem \ref{invariant_intro}, asymptotically negligible and the diffusion behaves, in long time, as the underlying Ornstein-Uhlenbeck process. Since the cocycle property (\ref{cocycle_intro}) is no longer satisfied, we loss the structure of linear RDS. To prove this result, we use once-again Proposition \ref{A} but we also need to apply \cite[Lemma 4.5]{MY} to the more general equivalent SDE (\ref{sde}). 

Following the terminology used in \cite{MY}, it is not difficult to see that this equation is asymptotically time-homogeneous and $S_\ast \Gamma$-ergodic, with $S$ the scale function of the Ornstein-Uhlenbeck diffusion process having $\Gamma\sim\mathcal N(0,1)$ as stationary distribution and $S_*\Gamma$ the pushforward distribution of $\Gamma$ by $S$.  As they mention in \cite{MY}, the main difficulty to apply this lemma is usually to show the boundedness in probability. To this end, we need to use  again the Foster-Lyapunov functions $U_\alpha$ and $V_\alpha$.

%%%%%%%%%%%%%%%%%%%%%%%%%%%%%%%%%%%%%%%%%%%%%%%%%%%%%%%%%%%%%%%%%%%%%%%%%%%%%%%%%%%%%%%%%%%%%%%%%%%%%%%%%%%%%%%%%%%%%%%%%%%%%%%%%%%%%%%%%%%%%%%%%%%%%%%%%%%%%%%%%%%%%%%%%%%%%%%%%%%%%%%%%%%%%%%%%%%%%%%%%%%%%%%%%%%%%%%%%%%%%%%%%%%%%%%%%%%%%%%%%%%%%%%%%%%%%%%%%%%%%%%%%%%%%%%%%%%%%%%%%%%%%%%%%%%%%%%%%%%%%%%%%%%%%%%%%%%%%%%%%%%%%%%%%%%%%%%%%%%%%%%%%%%%%%%%%%%%%%%%%%%%%%%%%%%%%%%%%%%%%%%%%%%%%%%%%%%%%%%%%%%%%%%%%%%%%%%%%%%%%%%%%%%%%%%%%%%%%%%%%%%%%%%%%%%%%%%%%%%%%%%%%%%%%%%%%%%%%%%%%%%%%%%%%%%%%%%%%%%%%%%%%%%%%%%%%%%%%%%%%%%%%%%%%%%%%%%%%%%%%%%%%%%%%%%%%%%%%%%%%%%%%%%%%%%%%%%%%%%%%%%%%%%%%%%%%%%%%%%%%%%%%%%%%%%%%%%%%%%%%%%%%%%%%%%%%%%%%%%%%%%%%%%%%%%%%%%%%%%%%%%%%%%%%%%%%%%%%%%%%%%%%%%%%%%%%%%%%%%%%%%%%%%%%%%%%%%%%%%%%%%%%%%%%%%%%%%%%%%%%%%%%%%%%%%%%%%%%%%%%%%%%%%%%%%%%%%%%%%%%%%%%%%%%%%%%%%%%%%%%%%%%%%%%%%%%%%%%%%%%%%%%%%%%%%%%%%%%%%%%%%%%%%%%%%%%%%%%%%%%%%%%%%%%%%%%%%%%%%%%%%%%%%%%%%%%%%%%%%%%%%%%%%%%%%%%%%%%%%%%%%%%%%%%%%%%%%%%%%%%%%%%%%%%%%%%%%%%%%%%%%%%%%%%%%%%%%%%%%%%%%%%%%%%%%%%%%%%%%%%%%%%%%%%%%%%%%%%%%%%

\section{Application to time-inhomogeneous Brox's  diffusions}
\label{sec:3}
\setcounter{equation}{0}

\subsection{Associated models}

We turn now to our main application, the study of the socalled time-inhomogeneous Brox's diffusion. We consider, for any $\beta\in\mathbb R$, the informal SDE driven by a standard Brownian motion $B$, independent of the Brownian environment  $(\Theta,\mathcal B,\mathcal W)$,
\begin{equation}\label{broxinhom_intro}
\dd Y_t= \dd B_t -\frac{1}{2}\frac{\theta^\prime(Y_t)}{t^{\beta}}\,\dd t,\quad Y_{u}=y\in\mathbb R,\; t\geq u>0,\;\theta\in\Theta.
\end{equation}
A weak solution to equation (\ref{broxinhom_intro}) is, in the same manner as in definition \ref{solution_intro}, the diffusion whose conditional generator, given $\theta\in\Theta$, is 
\begin{multline*}
\mathcal L_\theta:=\left[\frac{1}{2} e^{{\theta(x)}/{t^\beta}}\frac{\partial}{\partial x}\left(e^{-{\theta(x)}/{t^\beta}}\frac{\partial}{\partial x}\right)\right] + \frac{\partial }{\partial t},\quad\mbox{with}\\ 
D(\mathcal L_\theta):=
\left\{F(t,x)\in {\rm C}^{1} : e^{-{\theta(x)}/{t^{\beta}}}\partial_x F(t,x)\in {\rm C}^{1}\right\}.	
\end{multline*}
As for equation (\ref{diff_intro}), where we can assume  without loss of generality that $s=0$, we can assume that $u=1$ in equation (\ref{broxinhom_intro}). Moreover, as in (\ref{diff_my_intro}), we assume that $\beta=r+{1}/{4}$ and we 
define, for all continuous functions $\omega$ on $[1,\infty)$ and all measurable function $G$ on $[1,\infty)\times\mathbb R$, 
$\Phi_e(\omega)(t):={\omega(e^t)}/{e^{t/2}}$ and $\mathcal E G (t,x):=G(e^t,e^{t/2}x)$. 

It is a simple calculation to see that $\mathcal E : D(\mathcal L_\theta)\longrightarrow D(L_\theta)$ is a bijection and that
$L_\theta = \mathcal E\circ\mathcal L_\theta\circ \mathcal E^{-1}$. 
In the same way as in \cite[Proposition 2.1 and Section 2.2.1]{MY} we deduce that $\{Y_t : t\geq 1\}$ is a weak solution to equation (\ref{broxinhom_intro}) if and only if $\{Z_t:=\Phi_{\mathtt e}(Y_t) : t\geq 0\}$ is a weak solution to equation (\ref{diff_intro}). Then a direct application of Theorem \ref{existence_intro} gives that for all $\theta\in\Theta$, there exists a unique irreducible strongly Feller diffusion process  solution to equation (\ref{broxinhom_intro}). 

Let $\mathds Q_{y}(\theta)$ be its quenched distribution and denote by $\{R_t(\theta) : t\geq 1\}$, the time-inhomogeneous semigroup associated to $\{X_t/\sqrt{t} : t\geq 1\}$ under $\mathds Q_{y}(\theta)$, and by $\widehat {\mathds Q}_{y}$ and $\{\widehat R_t : t\geq 1\}$,  there annealed counterparts.

\subsection{Associated asymptotic behaviours}

The following two corollaries are the analogous of Theorems \ref{invariant_intro} and \ref{CG}.  We recall that $S_\lambda$ is defined in (\ref{scaling}).

\begin{coro} Assume that $\beta=1/4$. For all
 $\alpha\in(0,1)$ there exists $\lambda>0$ such that, for all $\nu \in\mathcal M_{1,U_{\alpha}}$ and $\hat \nu\in\mathcal M_{1,V_{\alpha}}$,
\begin{equation}\label{weakergodicity_intro_brox1}
\limsup_{t\to\infty}\frac{\log(\|\nu R_{t}(\theta) - \mu_{S_{\sqrt t}\,\theta}\|_{U_\alpha})}{\log t}\leq -\lambda\quad \mathcal W\mbox{-a.s.} 
\end{equation}
and 
\begin{equation}\label{weakergodicity_intro_brox2}
\lim_{t\to\infty}{\|\hat\nu \widehat R_{t} - \hat\mu \|_{V_\alpha}}=0.
\end{equation}
\end{coro}

\begin{coro} Assume that $\beta>1/4$. For any $y\in\mathbb R$ and for $\mathcal W$-almost all $\theta\in\Theta$, the following convergence holds under the quenched distribution $\mathds Q_y(\theta)$,
\begin{equation}\label{by3}
\lim_{t\to\infty} \frac{X_t}{\sqrt t} \overset{(\dd)}{=} \mathcal N(0,1).	
\end{equation}
\end{coro}
%~ \begin{coro} Assume that $\beta>1/4$. For all
 %~ $\alpha\in(0,1)$ there exists $\lambda>0$ such that, for all $\nu \in\mathcal M_{1,U_{\alpha}}$,
%~ \begin{equation}\label{by3}
%~ \lim_{t\to\infty}{\|\nu R_{t}(\theta) - \Gamma\|_{U_\alpha}}=0\quad \mathcal W\mbox{-a.s.} 
%~ \end{equation}
%~ \end{coro}

The scaling limits (\ref{weakergodicity_intro_brox1}), (\ref{weakergodicity_intro_brox2}) and  (\ref{by3}) are to be compared with the two convergences presented in (\ref{my1}) (the deterministic situation studied in \cite{MY}) and convergences (\ref{brox1}) and (\ref{brox2}) (the random time-homogeneous situation considered in \cite{Brox}). These results have some commons features with those presented in \cite{MY} and \cite{Brox} and also with those presented in \cite{schm,Boldri3,ZeiSz,Seppa,LOY,KO,KO2,Rhodes} concerning the quenched central limit theorem (\ref{by3}). There is still a phase transition phenomenon for $\beta=1/4$ and we obtain distinct quenched and annealed scaling limits for the critical point. Moreover, we are more accurate concerning the speed of convergence, which is polynomial here, and exponential in Theorem \ref{invariant_intro}.

Nevertheless, the case $\beta<1/4$ seems to be out of range of the present technics. In fact, we expect a stronger localization phenomenon and a subdiffusive behaviour of order $t^{2\beta}\log^2(t)$ when $\beta\geq 0$ and an almost sure convergence when $\beta<0$ (which can seen  as a generalization and mixture of results presented in (\ref{brox1}), (\ref{brox2}) and (\ref{my2})). Note that in the case where $\beta<0$, equation (\ref{broxinhom_intro}) is (via a simple change of time) a damped SDE in random environment. 

Furthermore, some methods elaborated in this paper can be used to study a similar interesting situation where we replace the Brownian environment $\theta$ in (\ref{broxinhom_intro}) by an another self-similar process. These situations are object of some works in progress. The case of a multiplicative noise or similar equations in higher dimension seems to be more difficult.

%%%%%%%%%%%%%%%%%%%%%%%%%%%%%%%%%%%%%%%%%%%%%%%%%%%%%%%%%%%%%%%%%%%%%%%%%%%%%%%%%%%%%%%%%%%%%%%%%%%%%%%%%%%%%%%%%%%%%%%%%%%%%%%%%%%%%%%%%%%%%%%%%%%%%%%%%%%%%%%%%%%%%%%%%%%%%%%%%%%%%%%%%%%%%%%%%%%%%%%%%%%%%%%%%%%%%%%%%%%%%%%%%%%%%%%%%%%%%%%%%%%%%%%%%%%%%%%%%%%%%%%%%%%%%%%%%%%%%%%%%%%%%%%%%%%%%%%%%%%%%%%%%%%%%%%%%%%%%%%%%%%%%%%%%%%%%%%%%%%%%%%%%%%%%%%%%%%%%%%%%%%%%%%%%%%%%%%%%%

\section{Preliminaries of Theorems \ref{existence_intro} and \ref{markov_feller_intro}}
\label{sec:4}
\setcounter{equation}{0}

\subsection{Linear perturbations of equation (\ref{diff_intro})}

We consider, for any $a\in\mathbb R$, the informal SDE 
\begin{equation}\label{diff_intro_g}
\dd Z_{t} = \dd B_{t} -\frac{1}{2} \partial_{x} Q_{\theta}(t,Z_{t}) \,\dd t,\quad	Z_{s}=z\in\mathbb R,\; t\geq s\geq 0,\; \theta\in\Theta,
\end{equation}
with the more general potential than (\ref{diff_intro}) given by
\begin{equation}\label{potential_intro_g}
Q_{\theta}(t,x):=a\frac{x^{2}}{2}+e^{-rt} T_{t}\theta(x)= V_{\theta}(t,x) +\frac{a-1}{2} x^{2}.
\end{equation}
Here once again $r\in\mathbb R$ and $B$ denotes a standard Brownian motion independent of the Wiener space $(\Theta,\mathcal B,\mathcal W)$. This equation coincides with equation (\ref{diff_intro}) for $a=1$. The conditional infinitesimal generator $A_{\theta}$ and its associated domains are given as in (\ref{ig_intro}) and (\ref{domains_intro}), replacing $V_{\theta}$ by $Q_{\theta}$. Moreover, it is not difficult to check that 
\begin{equation}\label{relation_ig}
A_{\theta}=L_{\theta} - \frac{a-1}{2}\, x\frac{\partial}{\partial {x}}.	
\end{equation}
We get that the domains of $A_{\theta}$ and $L_{\theta}$ are equals, in particular, the domains of $A_{\theta}$ do not depend on $a$. A weak solution to equation (\ref{diff_intro_g}) is, in the same way as in Definition \ref{solution_intro}, a solution to the martingale problem related to $(A_{\theta},D(A_{\theta}))$. In the sequel, we set 
\begin{equation*}\label{relation_ig}
A_{\theta,t} := A_{\theta} - \frac{\partial}{\partial {t}}.
%~ \frac{1}{2}e^{Q_\theta(t,x)}\frac{\partial }{\partial x}\left(e^{-Q_\theta(t,x)}\frac{\partial }{\partial x}\right)=	
\end{equation*}

%%%%%%%%%%%%%%%%%%%%%%%%%%%%%%%%%%%%%%%%%%%%%%%%%%%%%%%%%%%%%%%%%%%%%%%%%%%%%%%%%%%%%%%%%%%%%%%%%%%%%%%%%%%%%%%%%%%%%%%%%%%%%%%%%%%%%%%%%%%%%%%%%%%%%%%%%%%%%%%%%%%%%%%%%%%%%%%%%%%%%%%%%%%%%%%%%%%%%%%%%%%%%%%%%%%%%%%%%%%%%%%%%%%%%%%%%%%%%%%%%%%%%%%%%%%%%%%%%%%%%%%%%%%%%%%%%%%%%%%%%%%%%%%%%%%%%%%%%%%%%%%%%%%%%%%%%%%%%%%%%%%%%%%%%%%%%%%%%%%%%%%%%%%%%%%%%%%%%%%%%%%%%%%%%%%%%%%%%%%%%%%%%%%%%%%%%%%%%%%%%%%%%%%%%%%%%%%%%%%%%%%%%%%%%%%%%%%%%%%%%%%%%%%%%%%%%%%%%%%%%%%%%%%%%%%%%%%%%%%%%%%%%%%%%%%%%%%%%%%%%%%%%%%%%%%%%%%%%%%%%%%%%%%%%%%%%%%%%%%%%%%%%%%%%%%%%%%%%%%%%%%%%%%%%%%%%%%%%%%%%%%%%%%%%%%%%%%%%%%%%%%%%%%%%%%%%%%%%%%%%%%%%%%%%%%%%%%%%%%%%%%%%%%%%%%%%%%%%%%%%%%%%%%%%%%%

\subsection{Equivalent SDE and martingale problem}

We assume that $a>0$ and we introduce an auxiliary SDE on $\mathbb R$, which is naturally connected to equation (\ref{diff_intro_g}). Let $S_{}$ and $H_{}$ be the functions on $\Theta\times \mathbb R^2$ defined  by
\begin{multline}\label{pseudoscale}
S_{\theta}(t,x):=\int_0^x e^{Q_\theta(t,y)}{\dd y} = e^{-{t}/{2}}\int_0^{e^{{t}/{2}}x} \exp{\Big(a\frac{e^{-t}z^2}{2} - e^{-\left(r+{1}/{4}\right)t}\theta(z)\Big)}\dd z\\
\mbox{and}\quad S_{\theta}(t,H_{\theta}(t,x))=x.  
\end{multline}
Note that $H_{\theta}$ is well defined since $a>0$ and in this case, the socalled pseudo-scale function $x\longmapsto S_{\theta}(t,x)$ is an increasing bijection of $\mathbb R$. Moreover, by using the second representation of $S$, obtained by the change of time $z:=e^{{t}/{2}}y$, we can see that $S_{\theta}(t,x)$ and $H_{\theta}(t,x)$ are continuously differentiable with respect to $(t,x)\in\mathbb R^2$ and we can set
\begin{equation*}\label{coeff_sde}
\sigma_{\theta}(t,x):=(\partial_x S_{\theta})(t,H_{\theta}(t,x))\;\;\;\mbox{and}\;\;\; d_{\theta}(t,x):=(\partial_t S_{\theta})(t,H_{\theta}(t,x)).
\end{equation*}
In addition, remark that, for all $(\theta,s,t,x)\in\Theta\times\mathbb R^{3}$,
\begin{multline}\label{cocycle_coeff}
S_{\theta}(s+t,x) = S_{(e^{-rs}T_s\theta)}(t,x),\quad 
H_{\theta}(s+t,x)  = H_{(e^{-rs}T_s\theta)}(t,x),\\
\sigma_{\theta}(s+t,x) =  \sigma_{(e^{-rs} T_s\theta)}(t,x) \quad\mbox{and}\quad  d_{\theta}(s+t,x) = d_{(e^{-rs} T_s\theta)}(t,x).
\end{multline}

We can  consider, for any $\theta\in\Theta$, the  SDE on $\mathbb R$ with continuous coefficients and  driven by a standard Brownian motion $B$, independent of $(\Theta,\mathcal B,\mathcal W)$,
\begin{equation}\label{sde}
\dd {\widetilde Z}_t=\sigma_{\theta}(t,{\widetilde Z}_t)\,\dd B_t + d_{\theta}(t,{\widetilde Z}_t)\,\dd t,\quad {\widetilde Z}_s={\tilde z}\in\mathbb R,\quad t\geq s\geq 0.
\end{equation}
Let  ${\rm C}^{1,2}$ be the space of continuous functions $F(t,x)$ on $[s,\infty)\times \mathbb R$ such that $\partial_t F$, $\partial_x F$ and $\partial_{xx}^2 F$ exist and are continuous functions and  introduce
\begin{equation*}
{\widetilde A}_{\theta} := {\widetilde A}_{\theta,\,t} + \frac{\partial }{\partial t} := \left[\frac{\sigma_{\theta}^2(t,x)}{2}\frac{\partial^2}{\partial x^2} + d_{\theta}(t,x)\frac{\partial }{\partial x}\right]+\frac{\partial}{\partial t}.
\end{equation*} 
Note that $S_{\theta}$ and $H_{\theta}$ induce two bijections from the space of measurable functions on $[s,\infty)\times \mathbb R$ into itself, inverse to each other, by setting 
\begin{equation*}\label{bijection}
\mathcal S_{\theta} F (t,x):=F(t, S_{\theta}(t,x))\quad\mbox{and}\quad \mathcal H_{\theta} F (t,x):=F(t, H_{\theta}(t,x)).
\end{equation*} 
By restriction, we get that  $\mathcal S_{\theta}$ and $\mathcal H_{\theta}$ induce bijections 
\begin{multline*}
\mathcal S_{\theta} :{\rm C^{1,2}} \longrightarrow D(A_{\theta}),\quad \mathcal H_{\theta} : D(A_\theta)\longrightarrow {\rm C^{1,2}},\\
\mathcal S_{\theta} :{\rm W}^{1,2,\infty}_{\rm loc} \longrightarrow \overline D(A_\theta)\quad\mbox{and}\quad \mathcal H_{\theta} : \overline D(A_\theta)\longrightarrow {\rm W}^{1,2,\infty}_{\rm loc},
\end{multline*}
where ${\rm W}^{1,2,\infty}_{\rm loc}$ denote the Sobolev space  of continuous functions $F(t,x)$ on $[s,\infty)\times \mathbb R$ such that the partial derivatives $\partial_t F$, $\partial_x F$, $\partial_t(\partial_x F)$ and $\partial_{xx}^2 F$ exist and are locally bounded functions. Moreover, the infinitesimal generators $A_\theta$ and ${\widetilde A}_{\theta}$ are equivalent. More precisely, they satisfy
\begin{equation}\label{equivalent_op}
{\mathcal S_{\theta}}^{-1}\circ A_\theta \circ {\mathcal S_{\theta}} =  {\widetilde A}_{\theta}.\\
\end{equation}

\begin{prop}\label{equivalent_mp} For any $r\in\mathbb R$, $\theta\in\Theta$, $s\geq 0$ and $z,{\tilde z}\in\mathbb R$ such that ${\tilde z}:=S_{\theta}(s,z)$, $\{Z_{t} : t\geq s\}$ is a weak solution to equation (\ref{diff_intro_g}) if and only if $\{{\widetilde Z}_t:=S_{\theta}(t,Z_t):  t\geq s\}$ is a weak solution, up to the explosion time $\tau_{e}$, to SDE (\ref{sde}). Furthermore, there exists a unique weak solution $(\widetilde Z,B)$ and, for all $G\in {\rm W}^{1,2,\infty}_{\rm loc}$ and $s\leq t< \tau_e$,
\begin{equation}\label{Ito_sde}
G(t,{\widetilde Z}_t)=G(s,{\tilde z}) + \int_s^{t} {\widetilde A}_{\theta} G(u,{\widetilde Z}_u)\,\dd u \\
+\int_s^{t}\partial_x G(u,{\widetilde Z}_u)\,\sigma_{\theta}(u,{\widetilde Z}_u)\, \dd B_u.
\end{equation}
\end{prop}

\begin{proof} 
Assume that ${\widetilde Z}$ is a weak solution to (\ref{sde}). By using the Ito formula, $\widetilde Z$ solves the martingale problem related to $({\widetilde A}_\theta,{\rm C^{1,2}})$. Therefore, ${\widetilde Z}_s={\tilde z}$ and there exists an increasing sequence of stopping time $\{\tau_n : n\geq 0\}$ such that, for all $n\geq 0$ and $G\in{\rm C}^{1,2}$,  
\begin{equation*}
G(t\wedge \tau_n,{\widetilde Z}_{t\wedge \tau_n})-\int_s^{t\wedge \tau_n} {\widetilde A}_{\theta} G(u, {\widetilde Z}_{u})\,\dd u,\quad t\geq s,
\end{equation*}
is a local martingale, with 
\begin{equation*}
\tau_e:=\sup_{n\geq 0}\inf\big\{t\geq s : |{\widetilde Z}_t|\geq n\big\}=\sup_{n\geq 0}\tau_n.
\end{equation*}
We deduce from relation (\ref{equivalent_op}) that $\{Z_t:=H_\theta(t,{\widetilde Z}_t) : t\geq s\}$ is a weak solution to (\ref{diff_intro_g}) since $Z_s=z$, for all $n\geq 0$ and $F\in D(L_\theta)$, $G:= \mathcal H_\theta F\in {\rm C}^{1,2}$, and
\begin{equation*}
F(t\wedge \tau_n,Z_{t\wedge \tau_n})-\int_s^{t\wedge \tau_n}  A_\theta F(u, Z_u)\,\dd u\\ 
= G(t\wedge \tau_n,{\widetilde Z}_{t\wedge \tau_n})-\int_s^{t\wedge  \tau_n} {\widetilde A}_\theta G(u, {\widetilde Z}_u)\,\dd u.
\end{equation*}

A similar reasoning allow us to show that if $ Z$ is a weak solution to (\ref{diff_intro_g}) then $\{{\widetilde Z}_t:=S_\theta(t, Z_t):t\geq s\}$ is a weak solution to (\ref{sde}). Moreover, equation (\ref{sde}) has continuous coefficients $\sigma_\theta$ and $d_\theta$ and is strictly elliptic ($\sigma_\theta>0$) and we deduce, by using classical arguments of localization (see, for instance, \cite[pp. 250-251]{SV}), that there exists a unique weak solution $({\widetilde Z},B)$. Furthermore, by using the Ito-Krylov formula (see, for instance, \cite[Chapter  10]{Krylov} or \cite[p. 134]{Elworthy}), we obtain (\ref{Ito_sde}).
\end{proof}

%~ \begin{rema} We can adapt this construction and its consequences when $a\leq 0$. In this case, the pseudoscale function $x\longmapsto S_{\theta}(t,x)$ is then an increasing bijection from $\mathbb R$ to $E_{t}\subseteq \mathbb R$, an open subset which may depend on time, the SDE (\ref{sde}) on $\mathbb R$ becomes a SDE on a moving domain $\{E_{t} : t\geq s\}$ and the explosion time is replace by the exit time of this moving domain.
%~ \end{rema}

%%%%%%%%%%%%%%%%%%%%%%%%%%%%%%%%%%%%%%%%%%%%%%%%%%%%%%%%%%%%%%%%%%%%%%%%%%%%%%%%%%%%%%%%%%%%%%%%%%%%%%%%%%%%%%%%%%%%%%%%%%%%%%%%%%%%%%%%%%%%%%%%%%%%%%%%%%%%%%%%%%%%%%%%%%%%%%%%%%%%%%%%%%%%%%%%%%%%%%%%%%%%%%%%%%%%%%%%%%%%%%%%%%%%%%%%%%%%%%%%%%%%%%%%%%%%%%%%%%%%%%%%%%%%%%%%%%%%%%%%%%%%%%%%%%%%%%%%%%%%%%%%%%%%%%%%%%%%%%%%%%%%%%%%%%%%%%%%%%%%%%%%%%%%%%%%%%%%%%%%%%%%%%%%%%%%%%%%%%%%%%%%%%%%%%%%%%%%%%%%%%%%%%%%%%%%%%%%%%%%%%%%%%%%%%%%%%%%%%%%%%%%%%%%%%%%%%%%%%%%%%%%%%%%%%%%%%%%%%%%%%%%%%%%%%%%%%%%%%%%%%%%%%%%%%%%%%%%%%%%%%%%%%%%%%%%%%%%%%%%%%%%%%%%%%%%%%%%%%%%%%%%%%%%%%%%%%%%%%%%%%%%%%%%%%%%%%%%%%%%%%%%%%%%%%%%%%%%%%%%%%%%%%%%%%%%%%%%%%%%%%%%%%%%%%%%%%%%%%%%%%%%%%%%%%%%%%%%%%%%%%%%%%%%%

\subsection{Chain rules and nonexplosion}
 
To construct Lyapunov functions for the infinitesimal generator $L_\theta$, or more generally for $A_{\theta}$ associated to (\ref{diff_intro_g}), we need to give the associated chain rules. 
For all $\theta\in\Theta$ and $\varphi\in{\rm W}^{1,\infty}_{\rm loc}$ (the space of real continuous functions such
that the partial derivatives in the sense of distributions exist and are locally bounded functions) define  
\begin{equation}\label{domainfunction}
F^{\varphi}_\theta(t,x):=\int_0^x \exp{\big[e^{-rt}\,T_t\theta(y)\big]} \varphi(t,y)\,\dd y\in  \overline D(A_\theta).
\end{equation}
By standard computations, we get the following chain rules
\begin{equation}\label{CR0}
A_{\theta,t} F^\varphi_\theta(t,x)=\frac{1}{2}\exp\big[{e^{-rt}\,T_t\theta(x)}\big]\,\big(\partial_x\varphi(t,x) - a x \varphi(t,x)\big),
\end{equation}
and 
\begin{multline}\label{CR}
\partial_t F^\varphi_\theta(t,x)= \frac{1}{2}\exp{\big[e^{-rt}\,T_t\theta(x)\big]} \,x \,\varphi(t,x) - \frac{1}{2} F_\theta^\varphi(t,x) \\ +\int_0^x \exp\big[{e^{-rt}\,T_t\theta(y)}\big]\Big(\partial_t\varphi(t,y)-\frac{y}{2}\partial_x\varphi(t,y) \Big)\,\dd y\\
-\Big(r+\frac{1}{4}\Big)\int_0^x \exp\big[{e^{-rt}\,T_t\theta(y)}\big]\big(e^{-rt}\,T_t\theta(y)\big)\varphi(t,y)\,\dd y.
\end{multline}

\begin{prop}\label{nonexplosion} Assume that $a>1$. For any $r\in\mathbb R$, $\theta\in\Theta$, $s\geq 0$ and $z\in\mathbb R$, any weak solution $Z$ to (\ref{diff_intro_g}) is global and, for all $T>s$ and $0<\beta<(a-1)/2$,
\begin{equation}\label{moment_exp_explo}
\mathds E\Big[\exp{\Big(\beta \sup_{s \leq t\leq T} Z_{t}^{2}\Big)}\Big]<\infty.	
\end{equation}
\end{prop}

\begin{proof} 
Let $0<\alpha<a-1$ and $U_\alpha$ be the function defined in (\ref{lyapunovfunctions_intro}) and set
\begin{equation*}\label{lyapunovnonexplosion}
U_{\theta,\alpha}(t,x):=1+\int_0^x \exp{\big[e^{-rt}\,T_t\theta(y)\big]} U_\alpha^\prime(y)\,\dd y \in \overline D(A_\theta).  
\end{equation*}  
We shall prove that $U_{\theta}$ is a Lyapunov function, in the sense that, for all $T>s$, there exists $\lambda>0$ such that, for all $0\leq t\leq T$ and $x\in\mathbb R$,
\begin{equation}\label{lyapunov_nonexplosion}
A_\theta U_{\theta,\alpha}(t,x) \leq \lambda U_{\theta,\alpha}(t,x) \quad\mbox{and}\quad \lim_{|x|\to\infty}\inf_{0\leq t\leq T} U_{\theta,\alpha}(t,x)=\infty.
\end{equation}

First note that the second relation in (\ref{lyapunov_nonexplosion}) is clear since $\lim_{|x|\to\infty}\theta(x)/x^2=0$. Moreover, by using (\ref{CR}) and (\ref{CR0}), we can see that
\begin{equation}\label{calcule1}
A_{\theta,t} U_{\theta,\alpha}(t,x)=-\frac{1}{2} \alpha(a-\alpha)\left(1-\frac{1}{(a-\alpha)x^2}\right)x^2 \exp{\big[e^{-rt}T_t\theta(x)\big]}\,U_\alpha(x)
\end{equation}
and
\begin{multline}\label{partialderivative_t_nonexplosion}
\partial_t U_{\theta,\alpha}(t,x)= \frac{1}{2}\alpha x^{2}\exp{\big[e^{-rt}\,T_t\theta(x)\big]} U_\alpha(x) -\frac{1}{2}(U_{\theta,\alpha}(t,x)-1)\\
 -\int_0^x \exp{\big[e^{-rt}\,T_t\theta(y)\big]} \Big(\frac{\alpha y^2 +1}{2}\Big) U_{\alpha}^{\prime}(y)\,\dd y \hskip1.2cm\\
-\Big(r+\frac{1}{4}\Big)\int_0^x  \exp\big[e^{-rt}\,T_t\theta(y)\big] \big(e^{-rt}\,T_t\theta(y)\big) U_\alpha^\prime(y)\,\dd y.
\end{multline}
In addition, since $0<\alpha<a-1$, we can write for $x$ sufficiently large,
\begin{equation}\label{calcule3}
-\frac{1}{2}\alpha(a-\alpha)\left(1-\frac{1}{(a-\alpha)x^2}\right) +\frac{1}{2}\alpha <0.
\end{equation}
Then we get from (\ref{calcule3}) and (\ref{calcule1}) that there exist $L_{1}>0$ and a compact set $C$ such that, for all $0\leq t\leq T$ and $x\in\mathbb R$, 
\begin{equation}\label{calcule4}
A_{\theta,t} U_{\theta,\alpha}(t,x) + \frac{1}{2} \alpha x^2 \exp\big[e^{-rt}\,T_t\theta(x)\big]\, U_\alpha(x)\leq L_{1}\,{\mathds 1}_{C}(x)\leq L_{1}\, U_{\theta,\alpha}(t,x).
\end{equation}
Besides, we can see that there exists $L_{2}>0$ such that, for all $0\leq t\leq T$ and $y\in\mathbb R$,
\begin{equation}\label{calcule2}
\frac{\alpha y^2+1}{2} + \Big(r+\frac{1}{4}\Big) e^{-rt}\,T_t\theta(y) \geq \frac{\alpha y^2}{4} -L_{2}.
\end{equation}
We deduce from (\ref{calcule2}), (\ref{calcule4}) and (\ref{partialderivative_t_nonexplosion}) that (\ref{lyapunov_nonexplosion}) is satisfied with $\lambda := L_{1}+ L_{2}$. By using a classical argument (see, for instance, \cite[Theorem 10.2.1]{SV}) we get that the explosion time  is infinite a.s. Furthermore, the right hand side of (\ref{lyapunov_nonexplosion}) implies that $\{ e^{-\lambda t} U_{\theta,\alpha}(t,Z_{t}) : s\leq t\leq T\}$ is a positive supermartingale. By using the  maximal inequality (obtain from the optional stopping theorem) we get that, for all $R>0$,
\begin{equation}\label{doob}
R\, \mathds P\left(\sup_{s\leq t\leq T} e^{-\lambda t} U_{\theta,\alpha}(t,Z_{t}) \geq R \right) \leq e^{-\lambda s} U_{\theta,\alpha}(s,z).
\end{equation} 
Besides, we can check that, for all $\beta<\alpha/2$, there exists $c>0$ such that, for all $s\leq t\leq T$ and $x\in\mathbb R$, $c\,U_{\theta,\alpha}(t,x)\geq \exp(\beta x^{2})$. Then by using (\ref{doob}), we obtain  
\begin{equation*}
\mathds P\Big(\sup_{s\leq t\leq T} Z_{t}^{2} \geq R\Big) \leq c\,e^{\lambda (T-s)} U_{\theta,\alpha}(s,z) \exp(-\beta R).
\end{equation*}
Since $\beta$ and $\alpha$ are arbitrary parameters satisfying $\beta<\alpha/2<(a-1)/2$, we deduce from the last inequality that (\ref{moment_exp_explo}) holds for any $\beta<(a-1)/2$.
\end{proof}

\section{Proof of Theorem \ref{existence_intro}}
\label{sec:5}
\setcounter{equation}{0}

Theorem \ref{existence_intro} will be a direct consequence of Theorem \ref{nonexplosion_f} below.

\begin{thmm}\label{nonexplosion_f}
For any $a,r\in\mathbb R$, $\theta\in\Theta$, $s\geq 0$ and $z\in\mathbb R$, there exists a unique global weak solution to equation (\ref{diff_intro_g}). Moreover, there exists a standard Brownian motion $B$ such that, for all $F\in \overline D(A_{\theta})$,
\begin{equation}\label{ito_diff_intro_g}
F(t ,Z_t) = F(s,z)+\int_s^{t} A_{\theta} F(u,Z_u)\,\dd u +\int_s^{t}\partial_x F(u,Z_u)\,\dd B_u,\quad t\geq s.\\
\end{equation}
\end{thmm}

\begin{proof} First of all, when $a>1$, the proof is a direct consequence of Propositions (\ref{nonexplosion}) and (\ref{equivalent_mp}). More generally than relation (\ref{relation_ig}), we note that, for any $a_{1},a_{2}\in\mathbb R$, 
\begin{equation*}
 A_{\theta}^{(1)} =  A_{\theta}^{(2)} - \frac{a_{1}-a_{2}}{2} x \frac{\partial }{\partial x}, 
\end{equation*}	
where $A^{(i)}$ denotes the infinitesimal generator associated to $a_{i}$, $i\in\{1,2\}$. By using this relation, it is not difficult to see that the Girsanov transformation induces, by localization, a linear bijection between the weak solutions associated to parameters  $a_{1}$ and $a_{2}$. Since for all $a_{2}>1$ there exists a unique weak solution, we obtain that, for all $a_{1}\leq 1$ there exists a unique weak solution. Therefore, to complete the proof, it suffices to show that there exists a global weak solution. Remark that since uniqueness holds for the martingale problems, any weak solution is a Markov process.

Let $a_{1}\leq 1<a_{2}$ be and consider for $a_{2}$ a global weak solution $(Z,W)$ on a given filtered probability space $(\Omega,\mathcal F,\mathds P_{2})$. We set $k:=(a_{2}-a_{1})/2$ and, for all $t\geq s$,
\begin{equation*}
 D_{t}:=\exp\left(\int_{s}^{t} k\, Z_{u}\,\dd W_{u} -\frac{1}{2}\int_{s}^{t} k^{2}\,Z_{u}^{2}\,\dd u\right)\quad\mbox{and}\quad B_{t}:=W_{t}-W_{s}-\int_{s}^{t} k\, Z_{u}\, \dd u.	
\end{equation*} 
By using the moment inequality (\ref{moment_exp_explo}) and the Novikov criterion, we can see that  $\{ D_{t} : s\leq t\leq s+T\}$ is a martingale for any $0<T<(a_{2}-1)/k^{2}$. The Girsanov theorem applies and $\{B_{t} : s\leq t\leq s+T\}$ is a standard Brownian motion under the probability measure $\mathds P_{1}$, defined by the Radon-Nykodym derivatives 
\begin{equation*}
\dd {\mathds P_{1}}_{|\mathcal F_{t}} := D_{t}\; \dd {\mathds P_{2}}_{|\mathcal F_{t}},\quad s\leq t\leq s+T. 	
\end{equation*}
Moreover, for all $F\in \overline D( A_{\theta}^{(2)})\equiv \overline D(A_{\theta}^{(1)})$ and $s\leq t\leq s+T$,
\begin{multline*}
F(t ,Z_t) = F(s,z)+\int_s^{t} A_{\theta}^{(2)} F(u,Z_u)\,\dd u +\int_s^{t}\partial_x F(u,Z_u)\,\dd W_u\\
= F(s,z)+\int_s^{t} A_{\theta}^{(1)} F(u,Z_u)\,\dd u +\int_s^{t}\partial_x F(u,Z_u)\,\dd B_u.
\end{multline*} 
Then $\{(Z_{t},B_{t}) : s\leq t\leq s+T\}$ is a weak solution, which does not explode, on the filtered probability space $(\Omega,\mathcal F,\mathds P_{1})$. Since the life time $T$ is independent on the initial state $(s,z)$, we deduce by using the Markov property that the unique weak solution associated to $a_{1}$ is global. This completes the proof.
\end{proof}

%%%%%%%%%%%%%%%%%%%%%%%%%%%%%%%%%%%%%%%%%%%%%%%%%%%%%%%%%%%%%%%%%%%%%%%%%%%%%%%%%%%%%%%%%%%%%%%%%%%%%%%%%%%%%%%%%%%%%%%%%%%%%%%%%%%%%%%%%%%%%%%%%%%%%%%%%%%%%%%%%%%%%%%%%%%%%%%%%%%%%%%%%%%%%%%%%%%%%%%%%%%%%%%%%%%%%%%%%%%%%%%%%%%%%%%%%%%%%%%%%%%%%%%%%%%%%%%%%%%%%%%%%%%%%%%%%%%%%%%%%%%%%%%%%%%%%%%%%%%%%%%%%%%%%%%%%%%%%%%%%%%%%%%%%%%%%%%%%%%%%%%%%%%%%%%%%%%%%%%%%%%%%%%%%%%%%%%%%%%%%%%%%%%%%%%%%%%%%%%%%%%%%%%%%%%%%%%%%%%%%%%%%%%%%%%%%%%%%%%%%%%%%%%%%%%%%%%%%%%%%%%%%%%%%%%%%%%%%%%%%%%%%%%%%%%%%%%%%%%%%%%%%%%%%%%%%%%%%%%%%%%%%%%%%%%%%%%%%%%%%%%%%%%%%%%%%%%%%%%%%%%%%%%%%%%%%%%%%%%%%%%%%%%%%%%%%%%%%%%%%%%%%%%%%%%%%%%%%%%%%%%%%%%%%%%%%%%%%%%%%%%%%%%%%%%%%%%%%%%%%%%%%%%%%%%%%%%%%%%%%%%%%%%%%%%%%%%%%%%%%%%%%%%%%%%%%%%%%%%%%%%%%%%%%%%%%%%%%%%%%%%%%%%%%%%%%%%%%%%%%%%%%%%%%%%%%%%%%%%%%%%%%%%%%%%%%%%%%%%%%%%%%%%%%%%%%%%%%

\section{Proof of Theorem \ref{markov_feller_intro}}
\label{sec:6}
\setcounter{equation}{0}

We first show that it suffices to prove the analogous theorem for the more general equivalent SDE (\ref{sde}) (see Theorem \ref{markov_feller_sde}). Thereafter, we prove this theorem.

Let ${\widetilde {\mathds P}}_{s,{\tilde z}}(\theta)$ be the distribution of the global weak solution to the SDE (\ref{sde}), which existence is stated in Proposition (\ref{markov_feller_sde}), and denote by ${\widetilde P}_\theta(s,{\tilde z};t,\dd x)$ and ${\widetilde P}_{s,t}(\theta)$ the associated transition kernels and Markov kernels.

\begin{thmm}\label{markov_feller_sde} For any $r\in\mathbb R$ and all $\theta\in\Theta$,  the  family $\{{\widetilde {\mathds P}}_{s,{\tilde z}}(\theta) : s\geq 0,\, {\tilde z}\in\mathbb R\}$ is strongly Feller continuous. Moreover, the associated time-inhomogeneous semigroups $\{{\widetilde P}_{s,t}(\theta): t\geq s\geq 0,\,\theta\in\Theta\}$ satisfy 
\begin{equation}\label{cocycle_sde}
{{\widetilde P}}_{s,s+t}(\theta)={{\widetilde P}}_{0,t}(e^{-rs}T_s\theta)\quad\mbox{and}\quad {{\widetilde P}}_{0,s+t}(\theta)={{\widetilde P}}_{0,s}(\theta){{\widetilde P}}_{0,t}(e^{-rs}T_s\theta).
\end{equation}
Besides, ${{\widetilde P}}_\theta(s,{\tilde z};t,\dd x)$ admits a density ${\tilde p}_\theta(s,{\tilde z};t,x)$, which is measurable with respect to  $(\theta,s,t,{\tilde z},x)$ on $\Theta\times \{t>s\geq 0\}\times \mathbb R^{2}$, and which satisfies the lower local Aronson estimate: for all $\theta\in\Theta$, $T>0$ and compact set $C\subset \mathbb R$, there exists $M>0$ such that, for all $0\leq s<t\leq T$ and ${\tilde z},x \in C$,
\begin{equation}\label{lowerlocalaronson_sde}
{\tilde p}_{\theta}(s,{\tilde z};t,x)\geq \frac{1}{\sqrt{M(t-s)}}\exp{\left(-M \frac{|{\tilde z}-x|^{2}}{t-s}\right)}. 
\end{equation}
\end{thmm}

Denote by ${\mathds P}_{s,z}(\theta)$ the distribution of the unique global weak solution to  the equation (\ref{diff_intro_g}), which is given in Theorem \ref{nonexplosion_f}, and by $P_\theta(s,z;t,\dd x)$ and $ P_{s,t}(\theta)$ the associated transition kernels and Markov kernels. Assume first that $\{{\widetilde P}_{s,{\tilde z}}(\theta) : s\geq 0,\, {\tilde z}\in\mathbb R\}$ is strongly Feller continuous. One get by using Proposition \ref{equivalent_mp} that, for all bounded measurable function $F$ on $[0,\infty)\times \mathbb R$, $t\geq s\geq 0$ and $z\in\mathbb R$, 
\begin{equation*}
{\mathds E}_{s,z}(\theta)[F(t,X_t)]=
{{\widetilde {\mathds E}}}_{s,S_\theta(s,z)}(\theta)[F(t,H_\theta(t,X_t))].
\end{equation*}
Since $S_\theta$ is continuous on $\mathbb R^2$, we deduce that $\{ {\mathds P}_{s,z}(\theta) : s\geq 0,\, z\in\mathbb R\}$ is also strongly Feller continuous. Secondly, assume that $\{{\widetilde P}_{s,t}(\theta) : t\geq s\geq 0\}$ satisfies relations (\ref{cocycle_sde}). We get from (\ref{cocycle_coeff}) and  Proposition \ref{equivalent_mp}  that, for all nonnegative function $F$ on $ \mathbb R$, $s,t\geq 0$ and $z\in\mathbb R$,
\begin{multline*}
 P_{s,s+t}(\theta) F(z)  = {{\widetilde P}}_{s,s+t}(\theta)[F(H_\theta(s+t,\ast))](S_\theta(s,z)) \\ 
={{\widetilde P}}_{0,t}(e^{-rs}T_s\theta)[F(H_{(e^{-rs}T_s\theta)}(t,\ast))](S_{(e^{-rs}T_s\theta)}(0,z)) =   {P}_{0,t}(e^{-rs}T_s\theta)F(z).
\end{multline*}
By using the Markov property, we obtain relations (\ref{presque_cocycle_intro}).
Finally, assume that the transition kernels ${\widetilde P}_\theta(s,{\tilde z};t,\dd x)$ admits a measurable density ${\tilde p}_\theta(s,{\tilde z};t,x)$  which satisfies the lower local Aronson estimate (\ref{lowerlocalaronson_sde}). Once again, Proposition \ref{equivalent_mp} applies and gives that ${  P}_\theta(s,z;t,\dd x)$ admits a density $ p$ such that  
\begin{equation*}
 p_\theta(s,z;t,x)= {\tilde p}_\theta(s,S_\theta(s,z);t,S_\theta(t,x)) e^{ Q_\theta(t,x)}.
\end{equation*}
Since $S_\theta$ is a locally Lipschitz function, we deduce that  $ p_\theta(s,z;t,\dd x)$ is also measurable and satisfies the lower local Aronson estimate. In particular, Theorem \ref{markov_feller_sde} implies Theorem \ref{markov_feller_intro}.  This ends the proof, excepted for Theorem \ref{markov_feller_sde}.\\

\begin{proof}[Proof of Theorem \ref{markov_feller_sde}]
Since equation (\ref{sde}) is strictly elliptic ($\sigma_\theta>0$) and has continuous coefficients, it is classical (see for instance \cite[Corollary 10.1.4]{SV}) that its unique weak solution is a strongly Feller continuous diffusion, which admits transition densities ${\tilde p}_\theta(s,{\tilde z}; t,x)$ measurable with respect to $(s,t,{\tilde z},x)\in\{t>s\geq 0\}\times \mathbb R^2$ for each $\theta\in\Theta$. Moreover, we can see that relations (\ref{cocycle_sde}) are direct consequences of the Markov property and of (\ref{cocycle_coeff}). We need to prove the measurability of $\tilde p$ on $\Theta\times \{t>s\geq 0\}\times \mathbb R^2$ and the lower local Aronson estimate (\ref{lowerlocalaronson_sde}). Set, for all ${\delta}\geq 0$,
\begin{equation*}
{{\widetilde P}}_{{\delta},\theta}(s,{\tilde z};t,\dd x):={\widetilde {\mathds P}}_{s,{\tilde z}}(\theta)(X_t\in \dd x,\,\tau_{\delta}(s)>t)={\widetilde {\mathds P}}_{s,{\tilde z}}^{({\delta})}(\theta)(X_t\in \dd x,\,\tau_{\delta}>t)
\end{equation*}
with
\begin{equation*}
\tau_{\delta}(s):=\inf\{t\geq s : |X_t|\geq {\delta}\}\wedge T.
\end{equation*}
Here ${\widetilde {\mathds P}}_{s,{\tilde z}}^{({\delta})}(\theta)$ denotes the distribution of the truncated diffusion process whose coefficients are  given  on $[s,\infty)\times \mathbb R$ by
 \begin{equation*}
d_{\theta}^{({\delta})}(t,x):=d_\theta(t\wedge T, (x\wedge {\delta})\vee -\delta)\quad\mbox{and}\quad 
 \sigma_{\theta}^{({\delta})}(t,x):=\sigma_\theta(t\wedge T, (x\wedge {\delta})\vee -{\delta}).
 \end{equation*} 
Then the fundamental solution $\tilde p_{\theta}^{({\delta})}$ of the associated partial differential equation (PDE) satisfies the local Aronson estimates. Indeed, even if the associated partial differential operator is not of divergence form, we can see that it is equivalent to a uniformly elliptic divergence type operator, with bounded coefficients, employing the change of scale defined on $[s,\infty)\times \mathbb R$ by
\begin{equation*}
k_{\theta}^{({\delta})}(t,x):=\int_0^x{\frac{1}{(\sigma_{\theta}^{({\delta})}(t,y))^{2}}\exp\left({2\int_0^y \frac{d_{\theta}^{({\delta})}(t,z)}{(\sigma_{\theta}^{({\delta})}(t,z))^{2}}\,\dd z}\right)\,\dd y}.
 \end{equation*} 
Therefore, the results in \cite{Aronbon} or \cite{Porper} apply, and the fundamental solution ${\tilde q}_{\theta}^{({\delta})}$ of the associated PDE satisfies the global Aronson estimates. Besides, since
\begin{equation*}
\tilde p_{\theta}^{({\delta})}(s,{\tilde z},t,x)= {\tilde q}_{\theta}^{({\delta})}\big(s,k_{\theta}^{({\delta})}(s,{\tilde z}),t,k_{\theta}^{({\delta})}(t,x)\big)\partial_x k_{\theta}^{({\delta})}(t,x)
\end{equation*}
and $k_{\theta}^{({\delta})}$ is locally Lipschitz, we get that $\tilde p_{\theta}^{({\delta})}$ satisfies the local Aronson estimates. 

Then, following exactly the same lines as the proof of \cite[Theorem II.1.3]{Stroock_Elliptic} in the time-homogeneous situation, we can prove that the kernel ${{\widetilde P}}_{{\delta},\theta}$ admits a density $\tilde p_{\delta,\theta}$ such that, for all $0<\eta<1$, there exists $M>0$ such that, for all $0\leq s<t\leq T$, $|\tilde z|\leq \eta\delta$, $|x|\leq \eta \delta $ and  $|t-s|\leq (\eta \delta)^2$,
 \begin{equation*}
\tilde p_{\delta,\theta}(s,{\tilde z},t,x) \geq \frac{1}{\sqrt{M(t-s)}}\exp\left(-M\frac{|x-\tilde z|^{2}}{t-s}\right).
 \end{equation*} 
Since $\tilde p\geq \tilde p_\delta$, we deduce that $\tilde p$ satisfies (\ref{lowerlocalaronson_sde}) by taking $\delta$ sufficiently large. 

It remains to prove the measurability of $\tilde p$. We shall apply \cite[Theorem 11.1.4]{SV}. Since $(\theta,s,x)\longmapsto \sigma_\theta(s,x)$ and $(\theta,s,x)\longmapsto d_\theta(s,x)$ are continuous on $\Theta\times \mathbb R^2$, we can see that, for all convergent sequence $\theta_{n}\longrightarrow \theta$ in $\Theta$ and all $T,R>0$, 
\begin{equation*}
\sup_{[0,T]\times [-R,R]} |\sigma_{\theta_{n}}-\sigma_{\theta}|+|d_{\theta_{n}}-d_{\theta}| \xrightarrow[n\to\infty]{}0.
\end{equation*}
We can ckeck that the assumptions of \cite[Theorem 11.1.4]{SV} are satisfied and we conclude that, for all convergent sequence $(s_{n},\tilde z_{n})\longrightarrow (s,\tilde z)$ in $[0,\infty)\times\mathbb R$ and all bounded continuous function $G$ on the canonical space $\Omega$,
\begin{equation*}
\widetilde {\mathds E}_{s_{n},\tilde z_{n}}(\theta_{n})[G]\xrightarrow[n\to\infty]{} \widetilde {\mathds E}_{s,\tilde z}(\theta)[G] 
\end{equation*}
We deduce that $(\theta,s,\tilde z)\longmapsto \widetilde {\mathds E}_{s,\tilde z}(\theta)[G]$ is continuous on $\Theta\times [0,\infty)\times \mathbb R$. 
In particular, the family of probability measures $\{\widetilde{\mathds P}_{s,\tilde z}(\theta) : s\geq 0,\,\tilde z\in\mathbb R,\,\theta\in\Theta \}$ is tight 
and we can see that, for all bounded measurable function $F$ on $\mathbb R$, $(\theta,s,\tilde z,t)\longmapsto \widetilde {\mathds E}_{s,\tilde z}(\theta)[F(X_t)]$ is measurable on $\Theta\times \{ t> s\geq 0\}\times \mathbb R$. To this end, assume furthermore that $F$ is $L$-Lipschitz. We can write, for all compact set $K$ of the canonical space $\Omega$,
\begin{multline*}
|\widetilde {\mathds E}_{s,\tilde z}(\theta)[F(X_t)]-\widetilde {\mathds E}_{s_0,\tilde z_0}(\theta_0)[F(X_{t_0})]|
\leq L\,\widetilde {\mathds P}_{s,\tilde z}(\theta)(\Omega\setminus K) 
+ L\,\widetilde{\mathds  E}_{s,\tilde z}(\theta)[\mathds 1_K |X_t-X_{t_0}|] 
\\
+ |\widetilde {\mathds E}_{s,\tilde z}(\theta)[F(X_{t_0})]-\widetilde {\mathds E}_{s_0,\tilde z_0}(\theta_0)[F(X_{t_0})]|.
\end{multline*}
By letting $(s,\tilde z,\theta,t)\longrightarrow (s_0,\tilde z_0,\theta_0,t_0)$ and by using the tightness of the family of probability measure $\{\widetilde{\mathds P}_{s,\tilde z}(\theta) : s\geq 0,\,\tilde z\in\mathbb R,\,\theta\in\Theta \}$, we get the continuity and we deduce our claim, since any measurable bounded function is the bounded pointwise limit of a sequence of Lipschitzian functions. 

Therefore, we can define the measure $\nu$ on the product measurable space $\Theta\times \mathbb R^4$ by setting, for all $B\in\mathcal B$ and $I_1,I_{2},I_{3},I_{4}\in\mathcal B(\mathbb R)$, 
\begin{equation*}
\nu\left(B\times \prod_{k=1}^4 I_k\right):= \int_{B\times  I_1\times I_{2}\times I_{3}}\widetilde P_\theta(s,\tilde z, t, I_4)\mathds 1_{t>s\geq 0}\,\mathcal W(\dd \theta)\,\dd s\,\dd \tilde z\,\dd t.
\end{equation*}  
By standard results on disintegration of measures,  the Radon-Nykodym derivative of $\nu$ with respect to $\mathcal W(\dd \theta)\,\dd s\,\dd \tilde z\,\dd t\,\dd x$, which is nothing but $\tilde p_\theta(s,\tilde z,t,x)$,  is measurable.
\end{proof}

%%%%%%%%%%%%%%%%%%%%%%%%%%%%%%%%%%%%%%%%%%%%%%%%%%%%%%%%%%%%%%%%%%%%%%%%%%%%%%%%%%%%%%%%%%%%%%%%%%%%%%%%%%%%%%%%%%%%%%%%%%%%%%%%%%%%%%%%%
%%%%%%%%%%%%%%%%%%%%%%%%%%%%%%%%%%%%%%%%%%%%%%%%%%%%%%%%%%%%%%%%%%%%%%%%%%%%%%%%%%%%%%%%%%%%%%%%%%%%%%%%%%%%%%%%%%%%%%%%%%%%%%%%%%%%%%%%%%%%%%%%%%%%%%%%%%%%%%%%%%%%%%%%%%%%%%%%%%%%%%%%%%%%%%%%%%%%%%%%%%%%%%%%%%%%%%%%%%%%%%%%%%%%%%%%%%%%%%%%%%%%%%%%%%%%%%%%%%%%%%%%%%%%%%%%%%%%%%%%%%%%%%%%%%%%%%%%%%%%%%%%%%%%%%%%%%%%%%%%%%%%%%%%%%%%%%%%%%%%%%%%%%%%%%%%%%%%%%%%%%%%%%%%%%%%%%%%%%%%%%%%%%%%%%%%%%%%%%%%%%%%%%%%%%%%%%%%%%%%%%%%%%%%%%%%%%%%%%%%%%%%%%%%%%%%%%%%%%%%%%%%%%%%%%%%%%%%%%%%%%%%%%%%%%%%%%%%%%%%%%%%%%%%%%%%%%%%%%%%%%%%%%%%%%%%%%%%%%%%%%%%%%%%%%%%%%%%%%%%%%%%%%%%%%%%%%%%%%%%%%%%%%%%%%%%%%%%%%%%%%%%%%%%%%%%%%%%%%%%%%%%%%%%%%%%%%%%%%%%%%%%%%%%%%%%%%%%%%%%%%%%%%%%%%%%%%%%%%%%%%%%%%%%%%%%%%%%%%%%%%%%%%%%%%%%%%%%%%%%%%%%%%%%%%%%%%%%%%%%%%%%%%%%%%%%%%%%%%%%%%%%%%%%%%%%%%%%%%%%%%%%%%%%%%%%%%%%%%%%%%%%%%%%%%%%%%%%%

\section{Preliminaries of Theorems \ref{invariant_intro} and \ref{CG}}
\label{sec:7}
\setcounter{equation}{0}

\subsection{Uniform affine approximations of the environment}

In the following, set for all $\gamma\in(0,1/2)$ and $\theta\in\Theta$,
\begin{equation}\label{holder}
{ H}_\gamma(\theta) := \sup_{n\geq 0} \;\frac{\|\theta^+\|_{\gamma,n} + \|\theta^-\|_{\gamma,n}}{L(n)},\quad \Theta_\gamma:=\{0<H_\gamma<\infty\},
\end{equation}
with, for all $n\geq 0$ and $x\in\mathbb R$,
\begin{equation}\label{LL}
\|\theta^\pm\|_{\gamma,n}:=\sup_{n\leq x<y\leq n+1}\,\frac{|\theta(\pm y)-\theta(\pm x)|}{|y-x|^\gamma}\quad\mbox{and}\\\quad
L(x):=\sqrt{1+\log(1+|x|)}.
\end{equation}
In addition, denote for all $\ve>0$ by $A_{\gamma,\epsilon}(\theta)$ (see Figure \ref{fig:2}) the piecewise linear approximation of $\theta$, associated to the subdivision $S_{\gamma,\ve}:=\{x_{n,k} : n\in \mathds Z,\,0\leq k\leq m_n\}$, defined by $m_{n} :=h_{n}^{-1}:=[{L^{{1}/{\gamma}}(n)}{\ve^{-1}}]+1\in\mathbb N$, $x_{n,k}:=n + k\, h_{n}$ and $x_{-n,k}:=-x_{n,k}$. Then introduce the random affine approximation $W_{\gamma,\ve}$ defined, for all $\theta\in\Theta_\gamma$ by
\begin{equation}\label{randomapprox}
W_{\gamma,\ve}(\theta):= A_{\gamma,\eta_{\gamma,\ve}(\theta)}(\theta),\quad\mbox{with }\;\eta_{\gamma,\ve}(\theta):=\Big(\frac{\ve}{H_\gamma(\theta)}\Big)^{{1}/{\gamma}},
\end{equation}
and set
\begin{equation}\label{normesup_derivative_approx}
\Delta_{\gamma,\ve}(\theta)(x):=\theta(x)-W_{\gamma,\ve}(\theta)(x)\quad\mbox{and}\quad D_{\gamma,\ve}(\theta):=\sup_{x\in\mathbb R} \frac{| W_{\gamma,\ve}^{\prime}(\theta)(x)|}{L^{{1}/{\gamma}}(x)}.
\end{equation}

\begin{figure}[!ht]
\centering
\includegraphics[height=6.5cm,width=15cm]{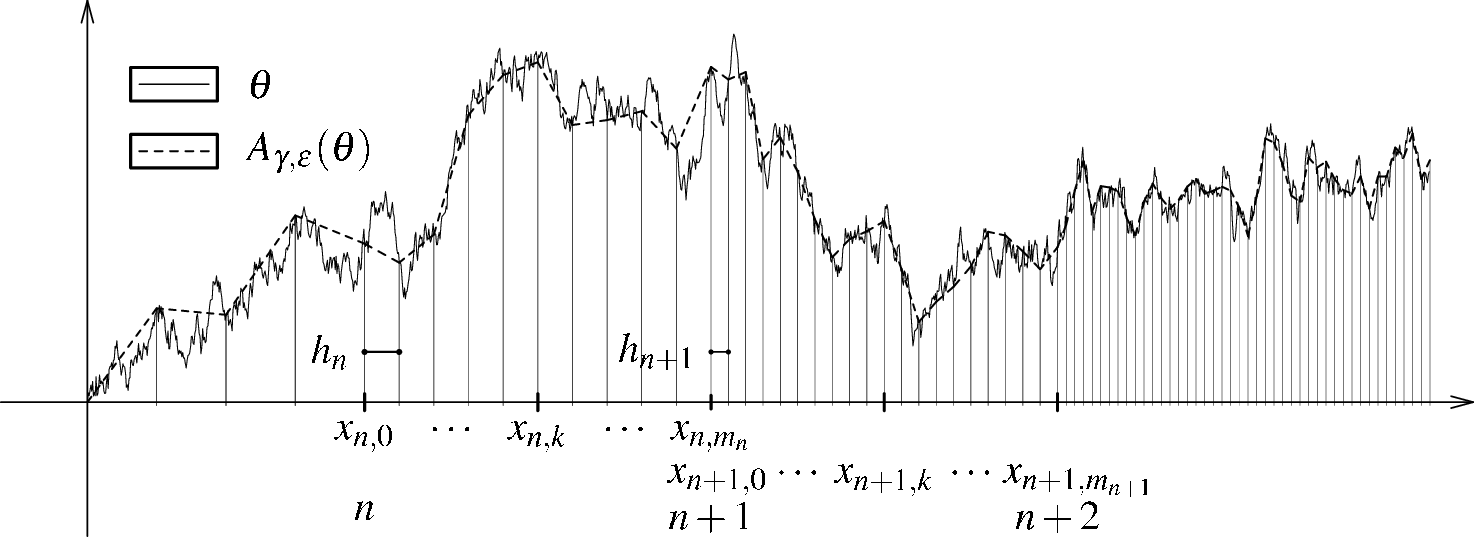}
\caption{Affine approximation of a typical Brownian path $\theta$}
\label{fig:2}
\end{figure}

\begin{prop}\label{uniform_random_approximation}
For all $\gamma\in(0,1/2)$, the subset $\Theta_\gamma\subset \Theta$ is $(T_t)$-invariant and of full measure. Furthermore, there exists $\alpha>0$ such that
\begin{equation}\label{moment_exp}
\mathds E_{\scriptscriptstyle \mathcal W}[\exp{(\alpha{ H}_\gamma^2)}]:=\int_{\Theta}\exp{(\alpha{ H}_\gamma^2(\theta))}\,\mathcal W(\dd\theta)<\infty.
\end{equation}
Besides, for all $\ve>0$ and $\theta\in\Theta_\gamma$,
\begin{equation}\label{holderrandomapprox} 
\sup_{x\in \mathbb R}|\Delta_{\gamma,\ve}(\theta)(x)| \leq \ve\quad\mbox{and}\quad D_{\gamma,\ve}(\theta)\leq \ve\big(1+(\ve^{-1}{H_\gamma(\theta)})^{{1}/{\gamma}}).
\end{equation}
\end{prop}

\begin{proof} 
Clearly $H_\gamma : \Theta\rightarrow [0,\infty]$ is a seminorm and to get inequality (\ref{moment_exp}) it suffices to apply the Fernique theorem presented in \cite[Theorem 1.3.2, p. 11]{Fernique}. To this end, we need to check that $\mathcal W(H_\gamma<\infty)>0$. By using the Hölder continuity of the Brownian motion on compact sets, the seminorm defined on $\Theta$ by $N(\theta):=\|\theta^+\|_{\gamma,1}+\|\theta^-\|_{\gamma,1}$ is finite $\mathcal W$-a.s. Moreover, by using the Fernique theorem and the Markov inequality, we deduce that there exists $c,\beta>0$ such that, for $r$ sufficiently large,
\begin{equation*}
F(r):=\mathcal W(\{N\geq r\})\leq \mathds E_{\scriptscriptstyle \mathcal W}[\exp(\beta N^2)]e^{-\beta r^2}\leq c e^{-\beta r^2}.
\end{equation*}
Besides, the random variables $(\theta\mapsto \|\theta^+\|_{\gamma,n}+\|\theta^-\|_{\gamma,n})$, $n\geq 0$, being i.i.d. by using again the Markov property, we get that
\begin{equation*}
\lim_{h\to\infty} \mathcal W(\{{\rm H}_\gamma \leq h\})=\lim_{h\to\infty}\prod_{n=0}^\infty \left(1- F\left(h\,L(n)\right)\right)\\
\geq \lim_{h\to\infty}\prod_{n=0}^\infty \Big(1-\frac{c}{n^{ \beta h^2}}\Big)=1.
\end{equation*}
Fernique's theorem applies and we deduce (\ref{moment_exp}). The fact that $\Theta_\gamma$ is $(T_t)$-invariant is obtained by noting that, for all $\theta\in\Theta$ and $t\in\mathbb R$,
\begin{equation*}
H_\gamma(T_t\theta)\leq 2 e^{\left(\gamma-{1}/{2}\right)({t}/{2})}\big(e^{{t}/{2}}+1\big)  \sup_{n\geq 0}\bigg[\frac{L\left((n+1)e^{{t}/{2}}+1\right)}{L(n)}\bigg] H_\gamma(\theta).
\end{equation*}
Furthermore, let $\ve>0$, $n\geq 0$ and $x,y\in\mathbb R$ be such that $n\leq  x,y\leq n+1$ and $|y-x|\leq h_{n}$, where $h_n$ denotes the step of the subdivision $S_{\gamma,\ve}$ defined in Figure \ref{fig:2}. We can see that
\begin{equation*}
|\theta^\pm(y)-\theta^\pm(x)|\leq L(n) H_\gamma(\theta)\; h_n^\gamma\leq {H_\gamma}(\theta)\,\ve^\gamma
\end{equation*}
and, when $|y-x|=h_n$, we get
\begin{equation*}
\frac{|\theta^\pm(y)-\theta^\pm(x)|}{|y-x|}\leq L(n) H_\gamma(\theta) |y-x|^{\gamma-1}\\
\leq H_\gamma(\theta)\,\ve^{\gamma}\,(\ve^{-1}{L^{{1}/{\gamma}}(n)}+1).
\end{equation*}
Therefore, we obtain that
\begin{equation*}
\sup_{x\in \mathbb R}|\theta(x)-A_{\gamma,\ve}(\theta)(x)|\leq {H}_{\gamma}(\theta)\,\ve^\gamma
\quad\mbox{and}\\\quad
\sup_{x\in\mathbb R} \frac{| A_{\gamma,\ve}^\prime(\theta)(x)|}{L^{{1}/{\gamma}}(x)}\leq {H}_{\gamma}(\theta)\,\ve^{\gamma-1}(1+\ve).
\end{equation*}
Replacing in the two last inequalities $\ve$ by $\eta_{\gamma,\ve}(\theta)$, defined in (\ref{randomapprox}), we deduce the proposition.
\end{proof}

%%%%%%%%%%%%%%%%%%%%%%%%%%%%%%%%%%%%%%%%%%%%%%%%%%%%%%%%%%%%%%%%%%%%%%%%%%%%%%%%%%%%%%%%%%%%%%%%%%%%%%%%%%%%%%%%%%%%%%%%%%%%%%%%%%%%%%%%%%%%%%%%%%%%%%%%%%%%%%%%%%%%%%%%%%%%%%%%%%%%%%%%%%%%%%%%%%%%%%%%%%%%%%%%%%%%%%%%%%%%%%%%%%%%%%%%%%%%%%%%%%%%%%%%%%%%%%%%%%%%%%%%%%%%%%%%%%%%%%%%%%%%%%%%%%%%%%%%%%%%%%%%%%%%%%%%%%%%%%%%%%%%%%%%%%%%%%%%%%%%%%%%%%%%%%%%%%%%%%%%%%%%%%%%%%%%%%%%%%%%%%%%%%%%%%%%%%%%%%%%%%%%%%%%%%%%%%%%%%%%%%%%%%%%%%%%%%%%%%%%%%%%%%%%%%%%%%%%%%%%%%%%%%%%%%%%%%%%%%%%%%%%%%%%%%%%%%%%%%%%%%%%%%%%%%%%%%%%%%%%%%%%%%%%%%%%%%%%%%%%%%%%%%%%%%%%%%%%%%%%%%%%%%%%%%%%%%%%%%%%%%%%%%%%%%%%%%%%%%%%%%%%%%%%%%%%%%%%%%%%%%%%%%%%%%%%%%%%%%%%%%%%%%%%%%%%%%%%%%%%%%%

\subsection{Random Foster-Lyapunov drift conditions}

\subsubsection{For the infinitesimal generators}

Let  $\varphi$ be a twice continuously differentiable  function from  $[1,\infty)$ into itself such that, $\varphi(v)=1$ on $[1,2]$, $\varphi(v)=v$ on $[3,\infty)$ and $\varphi(v)\leq v$ on $[1,\infty)$. In the sequel, we set 
\begin{equation}\label{lyapunovdomain}
F_{\theta}^{\gamma,\ve}(t,x):=1+\int_0^x \exp\left[{e^{-rt}\,T_t\Delta_{\gamma,\ve}(\theta)(y)}\right] U_\alpha^\prime(y)\, \dd y\in \overline D(L_\theta)
\end{equation}
and
\begin{multline}\label{lyapunovdomain2}
G_{\theta}^{\gamma,\ve}(t,x):=1+\int_0^x \exp\left[e^{-rt}\,{T_t\Delta_{\gamma,\ve}(\theta)(y)}\right] G^\prime_\alpha(y)\, \dd y\in\overline D(L_\theta),\\ \quad \mbox{with }\; G_\alpha(x):= \varphi(V_\alpha(x)).
\end{multline}
Here we use $G_\alpha=\varphi(V_\alpha)$  in (\ref{lyapunovdomain2}) instead of $V_\alpha$ because $V_\alpha^\prime$ do not belong to ${\rm W}^{1,\infty}_{loc}$ (there is a singularity in $0$) contrary to $U_\alpha^\prime$ in (\ref{lyapunovdomain}). 

\begin{lem}\label{step3} For all $r\in\mathbb R$, $\alpha\in(0,1)$, $\gamma\in(0,1/2)$, $T>0$ and $\lambda>0$, there exists $\overline \ve>0$ such that, for all $0<\ve<\overline \ve$, there exist a random variable $B: \Theta \longrightarrow [1,\infty)$ and $p,k,c>0$ such that, for all $\theta\in\Theta_\gamma$, $0\leq t\leq T$ and $x\in\mathbb R$,
\begin{equation}\label{lyapunovUalphastep3}
L_\theta F_{\theta}^{\gamma,\ve}(t,x) \leq -\lambda F_{\theta}^{\gamma,\ve}(t,x) + B_\theta,\\\quad 
\mbox{with }\; B_\theta\leq k \exp(c\,H_\gamma^p(\theta)).
\end{equation}
\end{lem}

\begin{proof} 
The proof will be a consequence of the following two steps.
 
\begin{step}\label{step11}
For all $0<\delta<1$ and $R\geq 1$, there exists $\ve_1>0$ such that, for all $0<\ve<\ve_1$ and $0<\ell<1$, there exist a map $R_1: \Theta \longrightarrow [R,\infty)$ and $c_1>0$ such that, for all $\theta\in\Theta_\gamma$, $0\leq t\leq T$ and $|x|\geq R_{1}(\theta)$,
\begin{equation}\label{step1}
L_{\theta,t} F_{\theta}^{\gamma,\ve}(t,x) \leq -\delta\frac{\alpha(1-\alpha)}{2} x^2 F_{\theta}^{\gamma,\ve}(t,x),\;\;\mbox{with}\; R_1(\theta)\leq c_1 (H_\gamma^{\frac{1}{\gamma(1-\ell)}}(\theta)\vee 1).
\end{equation}
\end{step}

First of all, by using chain rule (\ref{CR0}), with $a=1$, we obtain that
\begin{multline}\label{maj1}
L_{\theta,t} F_{\theta}^{\gamma,\ve}(t,x)  =  \frac{1}{2}\left(-\alpha(1-\alpha) x^2 -\alpha x\,e^{-rt}\,(T_tW_{\gamma,\ve}(\theta))^\prime(x) +\alpha\right) \\
\times \exp{\big[e^{-rt}\,T_t\Delta_{\gamma,\ve}(\theta)(x)\big]}\, U_\alpha(x),
\end{multline}
which can be written
\begin{multline}\label{calcul4}
L_{\theta,t} F_{\theta}^{\gamma,\ve}(t,x) = -\frac{1}{2}\alpha(1-\alpha)  \left[1 -\frac{1}{(1-\alpha)x^2} +\frac{e^{-rt}\,(T_tW_{\gamma,\ve}(\theta))^\prime(x)}{(1-\alpha)x}\right] \\ 
\times x^2\exp{\big[e^{-rt}\,T_t\Delta_{\gamma,\ve}(\theta)(x)\big]}\, U_\alpha(x).
\end{multline}
Moreover, we can see that
\begin{equation}\label{calcul3}
|(T_tW_{\gamma,\ve}(\theta))^\prime(x)|
\leq \varphi_\gamma(t) D_{\gamma,\ve}(\theta)L^{{1}/{\gamma}}(x),\\
\quad\mbox{with }\;
\varphi_\gamma(t):= \left(1+{t}/{2}\right)^{{1}/{2\gamma}}e^{{t}/{4}}.
\end{equation} 
Recall that $D_{\gamma,\ve}$ is defined in (\ref{normesup_derivative_approx}). In order to simplify our calculations, introduce 
\begin{equation}\label{not}
q:=1\vee e^{-(r+{1}/{4})T}\quad\mbox{and}\quad \Psi(\ve):=\exp[q\ve].
\end{equation}
Note that 
\begin{equation}\label{calcul0}
\left(\Psi(\ve)\right)^{-1} U_\alpha(x)\leq F_{\theta}^{\gamma,\ve}(t,x)\leq \Psi(\ve) U_\alpha(x). 
\end{equation}
Besides, we can choose $\ve_1>0$ and $D\geq R$ such that
\begin{equation}\label{piticalcule2}
\left(1-\frac{1}{(1-\alpha)D^2}-\frac{1\vee e^{-rT}}{(1-\alpha)D}\right)\left(\Psi(\ve_1)\right)^{-2}\geq \delta.
\end{equation}
Then we deduce the left hand side of (\ref{step1}) by using (\ref{piticalcule2}), (\ref{calcul0}), (\ref{calcul3}), (\ref{calcul4}) and by setting, for any $0< \ve < \ve_1$,
\begin{equation}\label{pititedef}
R_1(\theta):= \left[{ \varphi_\gamma(T)D_{\gamma,\ve}(\theta)}{c_{\gamma,\ell}}\vee 1\right]^{\frac{1}{1-\ell}} D^{\frac{1}{1-\ell}},\\ \quad
\mbox{with }\; c_{\gamma,\ell}:=\sup_{|x|\geq 1}\frac{L^{{1}/{\gamma}}(x)}{|x|^{\ell}}<\infty.
\end{equation}
Furthermore, the right hand side of (\ref{step1}) is obtained by using the right hand side of (\ref{holderrandomapprox}) and by choosing $c_1$ sufficiently large.

\begin{step}\label{step22}
For all $0<\delta<1$ and $R\geq 1$, there exists $\ve_2>0$ such that, for all $0<\ve<\ve_2$,  there exists a constant $R_2\geq R$ such that, for all $\theta\in\Theta_\gamma$, $0\leq t\leq T$ and $|x|\geq R_{2}$,
\begin{equation}\label{step2}
\partial_t F_{\theta}^{\gamma,\ve}(t,x) \leq (1-\delta) \frac{\alpha}{2}x^2 F_{\theta}^{\gamma,\ve}(t,x).
\end{equation}
\end{step}

By using chain rule (\ref{CR}), we get that
\begin{multline}\label{calculder}
\partial_t F_{\theta}^{\gamma,\ve}(t,x) = \frac{\alpha}{2} x^2 \exp{\left[e^{-rt}\,T_t\Delta_{\gamma,\ve}(\theta)(x)\right]} U_\alpha(x) -\frac{1}{2}\left(F_{\theta}^{\gamma,\ve}(t,x)-1\right) \\ 
-\int_0^x \frac{\alpha y^2+1}{2} \exp{\left[e^{-rt}\,T_t\Delta_{\gamma,\ve}(\theta)(y)\right]} U_\alpha^\prime(y)\,\dd y\\
-\Big(r+\frac{1}{4}\Big) \int_0^x \left(e^{-rt}\,T_t\Delta_{\gamma,\ve}(\theta)(y)\right) \exp{\left[e^{-rt}\,T_t\Delta_{\gamma,\ve}(\theta)(y)\right]} U_\alpha^\prime(y)\,\dd y.
\end{multline}
We can write, by integration by parts in the third term of the right hand side of (\ref{calculder}), 
\begin{multline}\label{expression3}
\partial_t F_{\theta}^{\gamma,\ve}(t,x) =\frac{\alpha}{2} x^2 \exp{\left[e^{-rt}\,T_t\Delta_{\gamma,\ve}(\theta)(x)\right]} U_\alpha(x) -\left(F_{\theta}^{\gamma,\ve}(t,x)-1\right)\\
-\frac{1}{2}\alpha x^2 F_\theta^{\gamma,\ve}(t,x) +  \int_0^x F_\theta^{\gamma,\ve}(t,y) \alpha y\,\dd y\\
-\Big(r+\frac{1}{4}\Big)\int_0^x \left(e^{-rt}\,T_t\Delta_{\gamma,\ve}(\theta)(y)\right) \exp{\left[e^{-rt}\,T_t\Delta_{\gamma,\ve}(\theta)(y)\right]} U_\alpha^\prime(y)\,\dd y.
\end{multline}
Besides, by using (\ref{calcul0}) and the left hand side of (\ref{holderrandomapprox}), we can see that
\begin{equation*}\label{majj}
\left|\int_0^x \left(e^{-rt}\,T_t\Delta_{\gamma,\ve}(\theta)(y)\right) \exp{\left[e^{-rt}\,T_t\Delta_{\gamma,\ve}(\theta)(x)\right]} U_\alpha^\prime(y)\,\dd y \right| \\
\leq q\ve\Psi^2(\ve) F_\theta^{\gamma,\ve}(t,x)
\end{equation*} 
and
\begin{equation*}\label{majj2}
\left|\int_0^x F_\theta^{\gamma,\ve}(t,y) \alpha y\,\dd y\right|\leq \Psi^2(\ve) F_\theta^{\gamma,\ve}(t,x).
\end{equation*}
We deduce from the two previous inequalities, (\ref{expression3}) and (\ref{calcul0}) that  
\begin{equation}\label{borne}
\partial_t F_{\theta}^{\gamma,\ve}(t,x) \leq \Big(\big[\Psi^2(\ve)-1\big]\frac{\alpha}{2} x^2 + \big[1+
\kappa q\ve\big]\Psi^2(\ve)\Big)F_\theta^{\gamma,\ve}(t,x),
\end{equation}
with $\kappa:=|r|+{1}/{4}$. Inequality (\ref{step2}) is then a simple consequence of (\ref{borne}) by taking $\ve_2>0$ and $R_2\geq R$ such that, for all $x\geq R_2$,
\begin{equation*}
\big[\Psi^2(\ve_2)-1\big]\frac{\alpha}{2} x^2 + \big[1+
q\kappa\ve_2\big]\Psi^2(\ve_2)\leq (1-\delta)\frac{\alpha}{2}x^2.
\end{equation*}

\noindent
{\bf Proof of Lemma \ref{step3}.} We deduce Lemma \ref{step3} from (\ref{step2}) and (\ref{step1}). Indeed, we can choose $0<\delta< 1$ and $R\geq 1$ such that 
\begin{equation}\label{final}
%~ \left(\frac{\delta}{2} \alpha(1-\alpha) -(1-\delta)\frac{\alpha}{2}\right)>0\quad\mbox{and}\\\quad
\left(\delta\frac{\alpha(1-\alpha)}{2} -(1-\delta)\frac{\alpha}{2}\right) R^2\geq \lambda.
\end{equation}
Then we get the left hand side of (\ref{lyapunovUalphastep3}) by using (\ref{final}) and by setting $\overline \ve:=\ve_1\wedge \ve_2$ and
\begin{equation}\label{BB}
B_\theta:=\sup_{|x|\leq R_1(\theta)\vee R_2,\, 0\leq t\leq T} L_\theta F_\theta^{\gamma,\ve}(t,x).
\end{equation}
 Moreover, by using inequalities (\ref{borne}), (\ref{calcul0}), (\ref{calcul3}) and  (\ref{maj1}), we can see that there exists $C>0$ such that
\begin{equation}\label{MM}
L_\theta  F_\theta^{\gamma,\ve}(t,x)\leq C\big(1+ D_{\gamma,\ve}(\theta) |x| L^{1/\gamma}(x)+x^2\big) U_\alpha(x).
\end{equation}  
We obtain the right hand side of (\ref{lyapunovUalphastep3}) by taking $p:=2/(\gamma(1-\ell))$, $k,c$ sufficiently large and by using (\ref{MM}), (\ref{BB}) and the right hand sides of (\ref{step1}) and (\ref{holderrandomapprox}). 
\end{proof}

\begin{lem}\label{step3bis} For all $r\in\mathbb R$, $\alpha\in(0,1)$, $\gamma\in(\alpha/2,1/2)$, $T>0$, $\ve>0$ and $\lambda>0$, there exist a random variable $B : \Theta \longrightarrow [1,\infty)$, $k,c>0$  and $0<p<2$ such that, for all $\theta\in\Theta_\gamma$, $0\leq t\leq T$ and $x\in\mathbb R$,
\begin{equation}\label{2step3}
L_\theta G_{\theta}^{\gamma,\ve}(t,x) \leq -\lambda G_{\theta}^{\gamma,\ve}(t,x) + B_\theta,\quad\mbox{ with }\;B_\theta\leq k \exp(c\,H_\gamma^p(\theta)).
\end{equation}
\end{lem}

\begin{proof} 
This proof uses similar ideas as the proof of Lemma \ref{step3} and we only give the main lines. Once again, the proof will be a consequence of the following two steps.

\begin{step2}\label{step21}
For all $0<\delta<1$, $R\geq 1$ and $0<\ell<1$, there exist $R_1 : \Theta \longrightarrow [R,\infty)$ and $c_1>0$ such that, for all $\theta\in\Theta_\gamma$, $0\leq t\leq T$ and $|x|\geq R_1(\theta)$,
\begin{equation}\label{step1bis}
L_{\theta,t} G_{\theta}^{\gamma,\ve}(t,x) \leq -(1-\delta)\frac{\alpha}{2} |x|^\alpha G_{\theta}^{\gamma,\ve}(t,x),\;\;\;\mbox{with }\; R_1(\theta)\leq c_1(H_\gamma^{\frac{1}{\gamma(1-\ell)}}(\theta)\vee 1).
\end{equation}
\end{step2}

By using chain rule (\ref{CR0}), with $a=1$,  we can see that, for all $x\in\{V_\alpha > 3\}$,
\begin{multline*}
L_{\theta,t} G_{\theta}^{\gamma,\ve}(t,x)
=  - \frac{\alpha}{2}  \left(1 +\frac{e^{-rt}\,(T_tW_{\gamma,\ve}(\theta))^\prime(x)}{x} -\frac{\alpha}{|x|^{2-\alpha}} + \frac{1-\alpha}{x^2}\right) \\
\times |x|^\alpha \exp{\left[e^{-rt}\,T_t\Delta_{\gamma,\ve}(\theta)(x)\right]}\, G_{\alpha}(x).\quad \label{exp3}
lien URL unsrt\end{multline*}
Moreover, we can choose $D\geq 1$ such that $\{V_\alpha \leq  3\}\subset [-D,D]$ and
\begin{equation*}
\left(1-\frac{1\vee e^{-rT}}{D}-\frac{\alpha}{D^{2-\alpha}}\right) (\Psi(\ve))^{-2}\geq (1-\delta).
\end{equation*}
Then by setting $R_1$ as in (\ref{pititedef}) we can deduce (\ref{step1bis}).

\begin{step2}\label{step23}
For all $\delta>0$ and $R\geq 1$, there exists a constant $R_2\geq R$  such that, for all $\theta\in\Theta_\gamma$, $0\leq t\leq T$ and $|x|\geq R_2$, 
\begin{equation}\label{step2bis}
\partial_t G_{\theta}^{\gamma,\ve}(t,x) \leq \delta \frac{\alpha}{2}|x|^\alpha G_{\theta}^{\gamma,\ve}(t,x).
\end{equation}
\end{step2}

By using chain rule (\ref{CR}) we can see that, for all $x\in\{V_{\alpha}>3\}$,
\begin{multline*}
\partial_t G_{\theta}^{\gamma,\ve}(t,x) = \frac{\alpha}{2} |x|^\alpha \exp{\big[e^{-rt}\,T_t\Delta_{\gamma,\ve}(\theta)(x)\big]} V_\alpha(x) - (G_{\theta}^{\gamma,\ve}(t,x) -1)\\ 
-\frac{1}{2}\int_0^x  \exp{\left[e^{-rt}\,T_t\Delta_{\gamma,\ve}(\theta)(y)\right]}\, y \,G_{\alpha}^{\prime\prime}(y)\,\dd y\\
-\Big(r+\frac{1}{4}\Big) \int_0^x \left[e^{-rt}\,T_t\Delta_{\gamma,\ve}(\theta)(y)\right] \exp{\left[e^{-rt}\,T_t\Delta_{\gamma,\ve}(\theta)(y)\right]} G_{\alpha}^\prime(y)\, \dd y. 
\end{multline*}
Then we can obtain (\ref{step2bis}) by using similar methods as in the proof of (\ref{step2}).\\

\noindent
{\bf Proof of Lemma \ref{step3bis}.} We deduce Lemma \ref{step3bis} from (\ref{step2bis}) and (\ref{step1bis}) in the same manner as we get Lemma \ref{step3} from (\ref{step2}) and (\ref{step1}). The main variation is that we need to choose $0<\ell<1$ in (\ref{step1bis}) such that $p:=\alpha/(\gamma(1-\ell))<2$.
\end{proof}

%%%%%%%%%%%%%%%%%%%%%%%%%%%%%%%%%%%%%%%%%%%%%%%%%%%%%%%%%%%%%%%%%%%%%%%%%%%%%%%%%%%%%%%%%%%%%%%%%%%%%%%%%%%%%%%%%%%%%%%%%%%%%%%%%%%%%%%%%
%%%%%%%%%%%%%%%%%%%%%%%%%%%%%%%%%%%%%%%%%%%%%%%%%%%%%%%%%%%%%%%%%%%%%%%%%%%%%%%%%%%%%%%%%%%%%%%%%%%%%%%%%%%%%%%%%%%%%%%%%%%%%%%%%%%%%%%%%%%%%%%%%%%%%%%%%%%%%%%%%%%%%%%%%%%%%%%%%%%%%%%%%%%%%%%%%%%%%%%%%%%%%%%%%%%%%%%%%%%%%%%%%%%%%%%%%%%%%%%%%%%%%%%%%%%%%%%%%%%%%%%%%%%%%%%%%%%%%%%%%%%%%%%%%%%%%%%%%%%%%%%%%%%%%%%%%%%%%%%%%%%%%%%%%%%%%%%%%%%%%%%%%%%%%%%%%%%%%%%%%%%%%%%%%%%%%%%%%%%%%%%%%%%%%%%%%%%%%%%%%%%%%%%%

\subsubsection{For the Markov kernels}

\begin{prop}\label{lyapunovUalpha} For all $r\in\mathbb R$, $\alpha\in(0,1)$, $\gamma\in(0,1/2)$ and  $\eta,\tau,T>0$, there exists a random variable $B : \Theta\longrightarrow [1,\infty)$ and $k,c,p>0$ such that, for all $\kappa>0$, $\theta\in\Theta_\gamma$, $0\leq s\leq t\leq T$ and $x\in\mathbb R$,
\begin{equation}\label{driftcondU}
P_{s,t}(\theta)U_\alpha(x)\leq (\eta+\kappa + \mathds 1_{s\leq t\leq s+\tau}) U_\alpha(x) + B_\theta\mathds 1_{ x\in \{U_\alpha \leq \kappa^{-1}B_\theta\}},
\end{equation}
with
\begin{equation}\label{lyapunovUalphaB}
B_\theta\leq k \exp(c\,H_\gamma^p(\theta)).
\end{equation}
\end{prop}

\begin{proof} 
Let $\lambda>e^q$ and $0<\overline\ve<1$ be as in Lemma \ref{step3} and $0<\ve<\overline \ve$ be such that $e^{-\lambda \tau+ 2q}\leq \eta$ and $e^{2q\ve}\leq \eta +1$, where $q$ is defined in (\ref{not}). One can see by using Ito's formula (\ref{ito_diff_intro}) that there exists a Brownian motion $W$ such that, under $\mathds P_{s,x}$,
\begin{multline}\label{Ito}
e^{\lambda t}F^{\gamma,\ve}_\theta(t,X_t) = e^{\lambda  s}F^{\gamma,\ve}_\theta(s,x) + \int_s^t e^{\lambda u}\left(L_\theta F^{\gamma,\ve}_\theta +\lambda F^{\gamma,\ve}_\theta)(u,X_u\right) \dd u \\ +\int_s^t e^{\lambda u} \partial_x F_\theta^{\gamma,\ve}(u,X_u) \dd W_u.
\end{multline}
Besides, we get from Lemma \ref{step3} that there exist a random variable $B : \Theta \longrightarrow [1,\infty)$, $k,c,p>0$ such that, for all $\theta\in\Theta_\gamma$, $0\leq s\leq t\leq T$ and $x\in\mathbb R$,
\begin{equation*}
L_\theta F_{\theta}^{\gamma,\ve}(t,x) \leq -\lambda F_{\theta}^{\gamma,\ve}(t,x) + B_\theta,\quad\mbox{with}\; B_\theta\leq k \exp(c\,H_\gamma^p(\theta)).
\end{equation*}
Then one can see by taking the expectation in (\ref{Ito}) and by using (\ref{calcul0}) that, for all $\theta\in\Theta_\gamma$, $0\leq s \leq t\leq T$ and $x\in\mathbb R$,
\begin{equation*}\label{lyapunovF}
 P_{s,t}(\theta) U_\alpha(x) \leq e^{-\lambda (t-s) + 2q\ve } U_\alpha(x) + \lambda^{-1}e^{q\ve} B_\theta\leq (\eta+\mathds 1_{s\leq t\leq s+\tau})  U_\alpha(x) + B_\theta.
\end{equation*}
and we deduce that inequalities (\ref{driftcondU}) and (\ref{lyapunovUalphaB}) hold for any $\kappa>0$.
%~ \begin{equation*}
%~ P_{s,t}(\theta) U_\alpha(x) \leq (\eta+\kappa+\mathds 1_{s\leq t\leq s+\tau}) U_\alpha(x) + 
 %~ B_\theta\mathds 1_{x\in \{U_\alpha \leq \kappa^{-1}B_\theta\}}
%~ \end{equation*}
%~ and this ends the proof.
\end{proof}

\begin{prop}\label{lyapunovValpha}
For all $r\in\mathbb R$, $\alpha\in(0,1)$, $\gamma\in(\alpha/2,1/2)$ and $\eta,\tau,T>0$, there exist a random variable $B : \Theta \longrightarrow [1,\infty)$, $k,c>0$ and $0<p<2$ such that, for all $\kappa>0$, $\theta\in\Theta_\gamma$, $0\leq s\leq t\leq T$ and $x\in\mathbb R$,
\begin{equation}\label{driftcondV}
P_{s,t}(\theta) V_\alpha(x)\leq (\eta+\kappa+\mathds 1_{s\leq t\leq s+\tau}) V_\alpha(x) + B_\theta \mathds 1_{x\in \{V_\alpha\leq \kappa^{-1} B_\theta\}},
\end{equation}
with
\begin{equation}\label{lyapunovValphaB}
B_\theta\leq k \exp(c\,H_\gamma^p(\theta)).
\end{equation}
\end{prop}

\begin{proof} 
The proof follows the same lines as the proof of Proposition \ref{lyapunovUalpha} and we only give the main ideas. Once again, by using Ito's formula and Lemma \ref{step3bis}, we can prove that there exist a random variable $B: \Theta \longrightarrow [0,\infty)$, $k,c>0$ and  $0<p<2$ such that, for all $\theta\in\Theta_\gamma$, $0\leq s\leq t\leq T$ and $x\in\mathbb R$, 
\begin{equation*}
P_{s,t}(\theta) G_\alpha(x) \leq (\eta +\mathds 1_{s\leq t\leq s+\tau}) G_\alpha(x)+ B_\theta,\quad\mbox{with}\;B_\theta\leq k \exp(c\,H_\gamma^p(\theta)).
\end{equation*}
Moreover, since $G_\alpha\leq V_\alpha$ and $G_\alpha(x)=V_\alpha(x)$, for $x\in\{V\geq 3\}$, we obtain that
\begin{multline*}
\mathds E_{s,x}(\theta)\left[V_\alpha(X_t)\mathds 1_{\{V_\alpha(X_t)\geq 3\}}\right]
%= \mathds E_{s,x}(\theta)\left[G_\alpha(X_t)\mathds 1_{\{V_\alpha(X_t)\geq 3\}}\right]\\
\leq (\eta +\mathds 1_{s\leq t\leq s+\tau}) G_\alpha(x)+ B_\theta\quad\mbox{and}\\
P_{s,t}(\theta) V_\alpha(x)\leq (\eta+\mathds 1_{s\leq t\leq s+\tau}) \leq (\eta+\mathds 1_{s\leq t\leq s+\tau}) V_\alpha(x) + (B_\theta+3).
\end{multline*}
This is enough to complete the proof.
\end{proof}

%%%%%%%%%%%%%%%%%%%%%%%%%%%%%%%%%%%%%%%%%%%%%%%%%%%%%%%%%%%%%%%%%%%%%%%%%%%%%%%%%%%%%%%%%%%%%%%%%%%%%%%%%%%%%%%%%%%%%%%%%%%%%%%%%%%%%%%%%%%%%%%%%%%%%%%%%%%%%%%%%%%%%%%%%%%%%%%%%%%%%%%%%%%%%%%%%%%%%%%%%%%%%%%%%%%%%%%%%%%%%%%%%%%%%%%%%%%%%%%%%%%%%%%%%%%%%%%%%%%%%%%%%%%%%%%%%%%%%%%%%%%%%%%%%%%%%%%%%%%%%%%%%%%%%%%%%%%%%%%%%%%%%%%%%%%%%%%%%%%%%%%%%%%%%%%%%%%%%%%%%%%%%%%%%%%%%%%%%%%%%%%%%%%%%%%%%%%%%%%%%%%%%%%%%%%%%%%%%%%%%%%%%%%%%%%%%%%%%%%%%%%%%%%%%%%%%%%%%%%%%%%%%%%%%%%%%%%%%%%%%%%%%%%%%%%%%%%%%%%%%%%%%%%%%%%%%%%%%%%%%%%%%%%%%%%%%%%%%%%%%%%%%%%%%%%%%%%%%%%%%%%%%%%%%%%%%%%%%%%%%%%%%%%%%%%%%%%%%%%%%%%%%%%%%%%%%%%%%%%%%%%%%%%%%%%%%%%%%%%%%%%%%%%%%%%%%%%%%%%%%%%%%%%%%%%%%%%%%%%%%%%%%%%%%%%%%%%%%%%%%%%%%%%%%%%%%%%%%%%%%%%%%%%%%%%%%%%%%%%%%%%%%%%%%%%%%%%%%%%%%%%%%%%%%%%%%%%%%%%%%%%%%%%%%%%%%%%%%%%%%%%%%%%%%%%%%%%%%%%%%%%%%%%%%%%%%%%%%%%%%%%%%%%%%%%%%%%%%%%%%%%%%%%%%%%%%%%%%%%%%%%%%%%%%%%%%%%%%%%%%%%%%%%%%%%%%%%%%%%%%%%%%%%%%%%%%%%%%%%%%%%%%%%%%%%%%%%%%%%%%%%%%%%%%%%%%%%%%%%%%%%%%%%%%%%%%%%%%%%%%%%%%%%%%%%%%%%%%%%%%

\subsection{Coupling method}

\subsubsection{Coupling construction}

We say that $C$ is a random $(1,\ve)$-coupling set associated to  the random Markov kernel $P$ and the random probability measure $\nu$ over $(\Theta,\mathcal B,\mathcal W)$ on $\mathbb R$, if $\ve : \Theta \longrightarrow (0,1/2]$ is a measurable map, $C_\theta$ is a compact set of $\mathbb R$ for $\mathcal W$-almost all $\theta\in\Theta$ and 
\begin{equation*}\label{coupling_set}
\inf_{z\in C_\theta}P_\theta(z;\ast) \geq \ve_\theta \nu_\theta(\ast)\quad\mathcal W\mbox{-a.s.}
\end{equation*}
Given a random $(1,\ve)$-coupling set $C$ associated to the random probability measure $\nu$, we construct a random Markov kernel $P^\star$ on $\mathbb R\times \mathbb R$ as follows. Let $\overline R$ and $ \overline P$ be two random Markov kernels on $\mathbb R\times \mathbb R$ satisfying, for all $x,y\in C_{\theta}$ and $A, B\in\mathcal B(\mathbb R)$,
\begin{equation*}
\overline R_\theta(x,y; A\times \mathbb R)=\frac{P_\theta(x;A)-\ve_\theta\,\nu_\theta(A)}{1-\ve_\theta},\quad 
\overline R_\theta(x,y; \mathbb R\times A)=\frac{P_\theta(y;A)-\ve_\theta\,\nu_\theta(A)}{1-\ve_\theta}
\end{equation*} 
and
\begin{equation}\label{couplingR}
\overline P_\theta(x,y; A\times B)=(1-\ve_\theta)\overline R_\theta(x,y; A\times B)+\ve_\theta\nu_\theta(A\cap B).
\end{equation}
Note that we can assume that $\overline P$ is a random coupling Markov kernel over $P$, in the sense that, for all $\theta\in\Theta$, $x,y\in\mathbb R$ and  $A\in\mathcal B(\mathbb R)$,
\begin{equation}\label{couplingeq}
{\overline P}_\theta(x,y; A\times \mathbb R)=P_\theta(x; A)\quad \mbox{and}\quad {\overline P}_\theta(x,y; \mathbb R \times A)=P_\theta(y; A).
\end{equation}
Then we define, 
\begin{equation}\label{coupling_construction}
P^\star_\theta(x,y; \ast):= 
\left\{\begin{array}{c}
\overline R_\theta(x,y;\ast),\quad \mbox{if}\; (x,y)\in C_{\theta}\times C_{\theta},\\
\overline P_\theta(x,y;\ast),\quad \mbox{if}\; (x,y)\notin C_{\theta} \times C_{\theta}.
\end{array}\right.
\end{equation}

\subsubsection{The Douc-Moulines-Rosenthal bound}

In order to simplify our claims, we set
\begin{multline*}
P_\theta:=P_1(\theta),\quad P_\theta(z;\dd x):=P_\theta(0,z;1,\dd x),\quad p_{\theta}(z,x):=p_{\theta}(0,z;1,x)\\
T\theta := T_1\theta\quad\mbox{and}\quad \overline U_\alpha(x,y):=\frac{U_\alpha(x)+U_\alpha(y)}{2}.
\end{multline*}
Moreover, we denote for any function $F : \Theta \longrightarrow (0,\infty)$, $n\in \mathbb N$ and $j\in\{0,\cdots,n\}$,
\begin{multline}\label{ergo_product}
F_{j,n}^+(\theta) := \max_{0\leq n_1 <\cdots < n_j\leq n-1} \prod_{k=1}^j 
F(e^{-r n_k}T^{n_k}\theta)\quad \mbox{and}\\\quad
F_{j,n}^-(\theta) := \max_{1\leq n_1 <\cdots < n_j\leq n} \prod_{k=1}^j 
F(e^{-r(n-n_k)}T^{-n_k}\theta) = F_{j,n}^{+}(T^{-n}\theta).
\end{multline}

\begin{prop}\label{A} 
For all $r\in\mathbb R$, $\alpha\in(0,1)$, $\gamma\in(0,1/2)$ and $\rho\in(0,\infty)$, there exist a random variable $B: \Theta \longrightarrow [1,\infty)$, with $\log(B)\in {\rm L}^1(\Theta,\mathcal B,\mathcal W)$, and a random $(1,\ve)$-coupling set $C$ over $(\Theta,\mathcal B,\mathcal W)$ on $\mathbb R$
 such that, for all $\theta\in\Theta_\gamma$,
\begin{equation}\label{assomp}
 P^\star_\theta \overline U_\alpha\leq \rho {\overline U}_\alpha + B_\theta\mathds 1_{ C_\theta\times C_\theta}\quad\mbox{and}\\\quad
 \sup_{(x,y)\in C_{\theta}\times C_\theta} \overline R_{\theta}\overline U_{\alpha}(x,y) \leq \frac{\rho B_\theta}{1-\ve_\theta}. 
\end{equation}
Moreover, for all $n\in\mathbb N$, $j\in\{1,\cdots,n+1\}$ and  $\nu_1,\nu_2\in\mathcal M_1$,
\begin{multline}\label{doucmoulines}
\|\nu_1 P_n(\theta) - \nu_2 P_n(\theta)\|_{U_\alpha} \leq   2\rho^n \, [(1-\ve)_{j,n}^+(\theta)\mathds 1_{j\leq n} + B_{j-1,n}^+(\theta)] \, \vert\vert \nu_1\vert\vert_{U_\alpha} \vert\vert \nu_2\vert\vert_{U_\alpha}\\
+2(1-\ve)_{j,n}^+(\theta)\mathds 1_{j\leq n}\sum_{k=0}^{n-1} \rho^k B(e^{-r (n-k-1)}\,T^{n-k-1}\theta).
\end{multline}
\end{prop}

\begin{proof} 
Let $\eta$ and $\kappa$ be two positive constants such that $\rho=\eta+2\kappa $ and use the  Proposition \ref{lyapunovUalpha} to obtain $\widetilde B:\Theta\longrightarrow [1,\infty)$ and $k,c,p>0$ such that, for all $\theta\in\Theta_\gamma$, 
\begin{equation*}
P_\theta U_\alpha \leq (\eta+\kappa)U_\alpha + \widetilde B_{\theta}\mathds 1_{C_{\theta}}, 
\,\,\mbox{with}\,\, \widetilde B_\theta\leq k \exp(c H_\gamma^p(\theta))\,\,\mbox{and}\,\, C_{\theta}=\{U_\alpha\leq \kappa^{-1}\widetilde B_\theta\}.
\end{equation*} 
The same arguments as in the proof of \cite[Proposition 11, p. 1660]{DoucMoulines} apply. Indeed, we can write, for any random Markov kernel $\overline P$ satisfying (\ref{couplingeq}),
\begin{equation*}
{\overline P}_\theta \overline U_\alpha \leq (\eta+\kappa) \overline U_\alpha + \frac{\widetilde B_{\theta}}{2} (\mathds 1_{C_{\theta}^{\mathtt c}\times C_{\theta}}+ \mathds 1_{C_{\theta}\times C_{\theta}^{\mathtt c}})+ \widetilde  B_\theta \mathds 1_{C_\theta\times C_\theta}.
\end{equation*}
Since $\widetilde B_{\theta}\leq 2 \kappa\overline U_{\alpha}$ on $C_{\theta}^{\mathtt c}\times C_{\theta}$ and $C_{\theta}\times C_{\theta}^{\mathtt c}$, we obtain from the last inequality that
\begin{equation}\label{maj}
{\overline P}_\theta \overline U_\alpha \leq \rho \overline U_\alpha + \widetilde  B_\theta \mathds 1_{C_\theta\times C_\theta}.
\end{equation}
Then we deduce that (\ref{assomp}) is satisfied by setting $B_\theta:=({(\rho \kappa^{-1}{\widetilde  B_\theta}+ \widetilde B_\theta)}{\rho^{-1}})\vee \widetilde B_\theta$ and by using  (\ref{maj}), (\ref{coupling_construction}) and (\ref{couplingR}). Besides, $\log(\widetilde B)\in{\rm L}^1(\Theta,\mathcal B,\mathcal W)$ by using (\ref{moment_exp}) and thus similarly for $\log(B)$. Moreover, for all $\theta\in\Theta_{\gamma}$, $C_{\theta}$ is a compact set and we get from the lower local Aronson estimate (\ref{lowerlocalaronson_intro}) that $C$ is a random $(1,\ve)$-coupling set associated to the random distribution $\nu$ defined, for all $\theta\in\Theta_{\gamma}$ and $A\in\mathcal B(\mathbb R)$ by
\begin{equation*}
\ve_{\theta}:=\left(\int_{\mathbb R} \inf_{z\in C_{\theta}}p_{\theta}(z,x)\,\dd x\right)\wedge \frac{1}{2}>0\quad\mbox{and}\quad \nu_{\theta}(A):= \frac{\int_{A} \inf_{z\in C_{\theta}}p_{\theta}(z,x)\,\dd x}{\int_{\mathbb R} \inf_{z\in C_{\theta}}p_{\theta}(z,x)\,\dd x}.
\end{equation*}
Furthermore, we can write by using (\ref{presque_cocycle_intro}) that
\begin{equation*}
P_{n}(\theta)=P(\theta)\cdots P(e^{-r(n-1)}\,T^{n-1}\theta)	
\end{equation*} 
and therefore, a direct application of \cite[Theorem 8, p. 1656]{DoucMoulines} gives (\ref{doucmoulines}).
\end{proof}

%%%%%%%%%%%%%%%%%%%%%%%%%%%%%%%%%%%%%%%%%%%%%%%%%%%%%%%%%%%%%%%%%%%%%%%%%%%%%%%%%%%%%%%%%%%%%%%%%%%%%%%%%%%%%%%%%%%%%%%%%%%%%%%%%%%%%%%%%%%%%%%%%%%%%%%%%%%%%%%%%%%%%%%%%%%%%%%%%%%%%%%%%%%%%%%%%%%%%%%%%%%%%%%%%%%%%%%%%%%%%%%%%%%%%%%%%%%%%%%%%%%%%%%%%%%%%%%%%%%%%%%%%%%%%%%%%%%%%%%%%%%%%%%%%%%%%%%%%%%%%%%%%%%%%%%%%%%%%%%%%%%%%%%%%%%%%%%%%%%%%%%%%%%%%%%%%%%%%%%%%%%%%%%%%%%%%%%%%%%%%%%%%%%%%%%%%%%%%%%%%%%%%%%%%%%%%%%%%%%%%%%%%%%%%%%%%%%%%%%%%%%%%%%%

\subsection{Ergodicity and exponential stability of the RDS}

\subsubsection{Ergodicity}

\begin{prop}\label{ergodicity}
The dynamical system $(\Theta,\mathcal B,\mathcal W,(T_t)_{t\in\mathbb R})$ is ergodic.
\end{prop}

\begin{proof} 
Introduce three measurable maps $U^\pm : \Theta \longrightarrow \Theta$ and $S_t : \Theta \longrightarrow \Theta$ defined by
\begin{equation*}
U^\pm(\theta):=(s\longmapsto e^{-{s}/{4}}\theta (\pm e^{s/2}))\quad\mbox{and}\quad S_t(\theta):=(s\longmapsto \theta(s+t)).
\end{equation*}
It is classical that the distribution of $U^\pm$ under the Wiener measure $\mathcal W$, denoted by $\Gamma$, is the distribution of the stationary Ornstein-Uhlenbeck process having the standard normal distribution as stationary distribution. This one is an ergodic process and, as a consequence, the dynamical system $(\Theta,\mathcal B,\Gamma, (S_t)_{t\in\mathbb R})$ is ergodic (see, for instance, \cite[Theorem 20.10]{Ka}).
Besides, it is clear that the following diagram is commutative:
\begin{equation*}\label{OU}
\xymatrix @R=1cm @C=3.5cm{
    (\Theta,\mathcal B, \mathcal W) \ar@<2pt>[r]^{U^\pm} \ar@<-2mm>[d]_{T_t}  & (\Theta,\mathcal B,\Gamma) \ar@<2mm>[d]^{S_t} \\
   (\Theta,\mathcal B, \mathcal W) \ar[r]_{U^\pm} \ar@<-2mm>[u]_{T^{-1}_t} & (\Theta,\mathcal B,\Gamma) \ar@<2mm>[u]^{S^{-1}_t}
  }
\end{equation*}
Let $A\in \mathcal B$ be such that $T^{-1}_t (A) =A$, with $t\neq 0$. By using  the ergodicity of the dynamical system $(\Theta,\mathcal B,\Gamma, (S_t)_{t\in\mathbb R})$, it follows that
\begin{equation*}
S^{-1}_t (U^\pm (A))= U^\pm (T_t^{-1}(A)) = U^\pm (A)\quad\mbox{and}\quad \Gamma(U^\pm (A))=0\quad\mbox{or}\quad =1.
\end{equation*}
Moreover, we can see that
\begin{equation*}
U^\pm (A)= (U^\pm)^{-1} (U^\pm (A))\quad\mbox{and}\quad  (U^+)^{-1} (U^+ (A)) \cap ( U^-)^{-1}(U^- (A))= A.
\end{equation*}
We conclude that $\mathcal W(A)=0 $ or $ =1$ and the proof is finished.
\end{proof}

\subsubsection{Exponential stability}

\begin{lem}\label{lemmaconv} Assume that $r=0$.  Let $F$ be such that $(\log(F)\vee 0)\in {\rm L}^{1}(\Theta,\mathcal B,\mathcal W)$ and $F^{\pm}$ as in (\ref{ergo_product}).
\begin{enumerate} 
\item If $\mathcal W(F<1)=1$ then, for all $L\geq 1$, there exists $\lambda>0$ such that
\begin{equation}
\limsup_{n\to\infty}\,e^{\lambda n}\,{F_{\left[\frac{n}{L}\right],n}^\pm}(\theta)=0\quad\mathcal W\mbox{-a.s.}
\end{equation}
\item If $\mathcal W(F\geq 1)>0$ then, for all $\eta>0$, there exists $L>0$ such that
\begin{equation}
\limsup_{n\to\infty}\,{e^{-\eta n}}\,{F_{\left[\frac{n}{L}\right],n}^\pm(\theta)}=0\quad\mathcal W\mbox{-a.s.}
\end{equation}
\end{enumerate}
\end{lem}

\begin{proof} 
We prove the lemma only for $F^+$ since the proof for $F^-$ is obtained replacing $\theta$ by $T^{-1}\theta$ and $T$ by $T^{-1}$. We set, for all $c\geq 0$ and $k\geq 1$,
\begin{equation*}
\log[F_{k}^{(c)}(\theta)]:={\log[F(T^{k-1}\theta)]\mathds 1_{F(T^{k-1}\theta)\geq c}}\quad\mbox{and}\quad F^{(c)}:= F^{(c)}_{1}.
\end{equation*}
Assume that $\mathcal W(F<1)=1$. We can see that there exist $0<c<1$ and $\ell>0$ such that
\begin{equation*}
\mathds E_{\scriptscriptstyle \mathcal W}[\mathds 1_{F\geq c}] <  {L^{-1}}\quad\mbox{and}\quad 
\mathds E_{\scriptscriptstyle \mathcal W}[\log (F^{(c)})]<-\ell.
\end{equation*}
By applying the ergodic theorem to the ergodic dynamical system $(\Theta,\mathcal B,\mathcal W,T)$ we obtain that, for $\mathcal W$-almost all $\theta\in\Theta$ and all integer $n$ sufficiently large,
\begin{equation*}
\sum_{k=1}^{n}\mathds 1_{F(T^{k-1}\theta)\geq  c}\leq  \left[\frac{n}{L}
\right]\quad\mbox{and}\quad {F_{\left[\frac{n}{L}\right],n}^+(\theta)}\leq{\prod_{k=1}^{n} F_{k}^{(c)}(\theta)}\leq {e^{-\ell n}}.
\end{equation*} 
Then we deduce the first point by taking $0<\lambda<\ell$. Assume that $\mathcal W(F\geq 1)>0$. Note that if $F$ is bounded $\mathcal W$-a.s. the second point of the lemma is obvious. Moreover, when $F$ is unbounded with positive probability, it is not difficult to see that there exist $0<\kappa<\eta$, $c\geq 1$ and $L\geq 1$ such that
\begin{equation*}
\mathds E_{\scriptscriptstyle \mathcal W}[\log(F^{(c)})]<\kappa\quad\mbox{and}\quad  \mathds E_{\scriptscriptstyle \mathcal W}[\mathds 1_{F\geq c}] > {L^{-1}}.
\end{equation*}
Once again, the ergodic theorem allow us to obtain the second point  since, for $\mathcal W$-almost all $\theta\in\Theta$ and all integer $n$ sufficiently large,
\begin{equation*}
\sum_{k=1}^{n}\mathds 1_{F(T^{k-1}\theta)\geq  c}\geq \left[\frac{n}{L}\right]\quad \mbox{and}\quad {F_{\left[\frac{n}{L}\right],n}^+(\theta)}\leq{\prod_{k=1}^{n} F_{k}^{(c)}(\theta)}\leq {e^{ \kappa n}}.
 \end{equation*}
\end{proof}

\begin{prop}\label{crucial} Assume that $r=0$. For all $\alpha\in(0,1)$ there exists $\lambda>0$ such that, for all families $\{\nu_t^\pm : t\geq 0\}$ of random distribution on $\mathbb R$ over $(\Theta,\mathcal B,\mathcal W)$ satisfying
\begin{equation}\label{hyp}
\lim_{t\to\infty}\frac{\log(\|\nu_t^\pm\|_{U_\alpha})}{t}=0\quad \mathcal W\mbox{-a.s.},
\end{equation}
the following discrete-time convergences hold:
\begin{equation}\label{conver0}
\limsup_{t\to\infty} \frac{\log(\|\nu_t^+(\theta) P_{[t]}(\theta) - \nu_t^-(\theta) P_{[t]}(\theta)\|_{U_\alpha})}{t}\leq -\lambda \quad \mathcal W\mbox{-a.s.}
\end{equation}
and
\begin{equation}\label{conver1}
\limsup_{t\to\infty} \frac{\log(\|\nu_t^+(\theta) P_{[t]}(T^{-[t]}\theta) - \nu_t^-(\theta) P_{[t]}(T^{-[t]}\theta)\|_{U_\alpha})}{t}\leq -\lambda \quad \mathcal W\mbox{-a.s.}
\end{equation}
\end{prop}

\begin{proof} 
We prove only (\ref{conver1}) since the proof of (\ref{conver0}) follows the same lines and employs the same arguments. Let $0<\rho<1$ be and, following Proposition \ref{A}, write that, for all $\theta\in\Theta_\gamma$, $t\geq 0$ and $j\in\{0,\cdots,[t]+1\}$,
\begin{multline}\label{doucmoulines2}
\| \nu_t^{+} P_{[t]}(T^{-[t]}\theta) - \nu_t^{-} P_{[t]}(T^{-[t]}\theta)\|_{U_\alpha} \\
\leq   2\rho^{[t]} \, [(1-\ve)_{j,[t]}^-(\theta)\mathds 1_{j\leq [t]} + B_{j-1,[t]}^-(\theta)] \, \vert\vert \nu_t^{+}\vert\vert_{U_\alpha} \vert\vert \nu_t^{-}\vert\vert_{U_\alpha}\\
+2(1-\ve)_{j,[t]}^-(\theta) \mathds 1_{j\leq [t]}\,\sum_{k=0}^{[t]-1} \rho^k B(T^{-k-1}\theta).
\end{multline}
Since $\log B\in{\rm L}^1(\Theta,\mathcal B,\mathcal W)$, the ergodic theorem allows us to see that, for all $\eta>0$,
\begin{equation*}
\lim_{k\to\infty}\frac{\log[B(T^{-k+1}\theta)]}{k}=0
\;\;\;\mbox{and}\;\;\; \limsup_{n\to\infty}\,{e^{-\eta n}} \, {\sum_{k=0}^{n-1} \rho^k B(T^{-k+1}\theta)}=0\quad\mathcal W\mbox{-a.s.}
\end{equation*}
Besides, one can see by using Lemma \ref{lemmaconv} that there exist $L\geq 1$ and $\ell>0$ such that 
\begin{equation*}
\lim_{n \to\infty}\,e^{-\eta n}\,B_{\left[\frac{n}{L}\right],n}^{-}(\theta) =0\quad\mbox{and}\quad \lim_{n \to\infty}\,e^{\ell n}\,(1-\ve)_{\left[\frac{n}{L}\right],n}^{-}(\theta) =0.
\end{equation*}
Therefore, we deduce from  (\ref{doucmoulines2}) the exponential convergence (\ref{conver1}).
\end{proof}

\section{Proof of Theorem \ref{invariant_intro}}
\setcounter{equation}{0}
\label{InvaSection}

Theorem \ref{invariant_intro} will be a consequence of Propositions \ref{weakergo} and \ref{annealconv}.  In the sequel, we introduce, for any operator $P$ acting on $\mathcal M_{F}$, $F\in \{U_\alpha,V_{\alpha}\}$, the subordinated norm 
\begin{equation*}
\|P\|_{F}:=\sup \{\|\mu P\|_{F} : \mu\in\mathcal M_{F},\, \|\mu\|_{F}\leq 1\}.
\end{equation*}

\subsection{Exponential weak ergodicity and quasi-invariant measure}

\begin{prop}\label{weakergo}
Assume that $r=0$. For all $\alpha\in(0,1)$ there exists $\lambda>0$ such that, for all $\nu_1,\nu_2\in \mathcal M_{1,U_\alpha}$,
\begin{equation}\label{conv0}
\limsup_{t\to\infty}\frac{\log\left(\|\nu_1 P_{t}(\theta) - \nu_2 P_{t}(\theta) \|_{U_\alpha}\right)}{t}\leq -\lambda\quad\mathcal W\mbox{-a.s.}
\end{equation}
Furthermore, there exists a unique (up to a $\mathcal W$-null set) random probability measure $\mu$ over $(\Theta,\mathcal B,\mathcal W)$ on $\mathbb R$ such that, for all $\alpha\in(0,1)$ there exists $\lambda>0$ such that, for all $\nu\in\mathcal M_{1,U_\alpha}$,
\begin{equation}\label{convinv}
\limsup_{t\to\infty} \frac{\log\big(\|\nu P_t(T_{-t}\theta)  -\mu_\theta \|_{U_\alpha}\big)}{t}\leq -\lambda \quad \mathcal W\mbox{-a.s.}
\end{equation}
Moreover, for all $t\geq 0$,
\begin{equation}\label{invconv}
\mu_\theta\in\mathcal M_{1,U_\alpha}\quad\mbox{and}\quad 
\mu_\theta P_t(\theta) = \mu_{T_t \theta}\quad \mathcal W\mbox{-a.s.}
\end{equation}
\end{prop}

\begin{proof} 
By using relation (\ref{cocycle_intro}) we can write 
$P_t(\theta)=P_{[t]}(\theta)P_{\{t\}}(T^{[t]}\theta)$ and we get 
\begin{equation}\label{conv2}
\left\|\nu_1 P_{t}(\theta) - \nu_2 P_{t}(\theta) \right\|_{U_\alpha} \leq   
\|\nu_1 P_{[t]}(\theta) - \nu_2 P_{[t]}(\theta)\|_{U_\alpha} \|P_{\{t\}}(T^{[t]}\theta)\|_{U_\alpha}.
\end{equation}
Moreover, by using Propositions \ref{lyapunovUalpha} and \ref{uniform_random_approximation} and the ergodic theorem,  we obtain 
\begin{equation}\label{conv3}
\lim_{n\to\infty} \frac{\log (\sup_{0\leq u\leq 1}\|P_u(T^n\theta)\|_{U_\alpha})}{n}=0\quad\mathcal W\mbox{-a.s.}
\end{equation}
Besides, a direct application of Proposition \ref{crucial} gives that there exists $\lambda>0$, independent of $\nu_1$ and $\nu_2$, such that 
\begin{equation}\label{conv1}
\limsup_{n\to\infty}\frac{\log\left(\|\nu_1 P_{n}(\theta) - \nu_2 P_{n}(\theta) \|_{U_\alpha}\right)}{n}\leq -\lambda\quad\mathcal W\mbox{-a.s.}	
\end{equation}
We deduce inequality (\ref{conv0}) from (\ref{conv1}), (\ref{conv3}) and (\ref{conv2}). Furthermore, one can see by using again Propositions \ref{crucial}, \ref{lyapunovUalpha} and \ref{uniform_random_approximation} and similar arguments that  
\begin{equation*}
\sum_{n=0}^\infty \|\nu P_{n+1}(T^{-n-1}\theta)-\nu P_{n}(T^{-n}\theta)\|_{U_\alpha}<\infty\quad \mathcal W\mbox{-a.s.}
\end{equation*}
We obtain that, for $\mathcal W$-almost all $\theta\in\Theta$, $\{\nu P_{n}(T^{-n}\theta) : n\geq 0\}$ is a Cauchy sequence in the separable Banach space $\mathcal M_{U_\alpha}$. We get that there exist $\lambda>0$ and a random probability measure $\mu_\theta\in\mathcal M_{U_\alpha}$ such that, for all $\nu\in \mathcal M_{1,U_\alpha}$, 
\begin{equation}\label{conv_discret}
\limsup_{n\to\infty} \frac{\log(\|\nu P_n(T^{-n}\theta)  -\mu_\theta \|_{U_\alpha})}{n}\leq -\lambda \quad \mathcal W\mbox{-a.s.}
\end{equation}
We deduce (\ref{convinv}) from (\ref{conv_discret}) in the same way as we obtain (\ref{conv0}) from (\ref{conv1}). Finally, (\ref{invconv}) is a consequence of (\ref{convinv}) and the cocycle property since
\begin{equation*}
\mu_\theta P_t(\theta)\overset{\mathcal M_{U_\alpha}}{=}\lim_{s\to\infty} \nu P_s(T_{-s}\theta) P_t(\theta) \overset{\mathcal M_{U_\alpha}}{=} \lim_{s\to\infty} \nu P_{t+s}(T_{-(s+t)}T_t\theta) \overset{\mathcal M_{U_\alpha}}{=} \mu_{T_t\theta}\quad\mathcal W\mbox{-a.s.}
\end{equation*}
\end{proof}

%%%%%%%%%%%%%%%%%%%%%%%%%%%%%%%%%%%%%%%%%%%%%%%%%%%%%%%%%%%%%%%%%%%%%%%%%%%%%%%%%%%%%%%%%%%%%%%%%%%%%%%%%%%%%%%%%%%%%%%%%%%%%%%%%%%%%%%%%%%%%%%%%%%%%%%%%%%%%%%%%%%%%%%%%%%%%%%%%%%%%%%%%%%%%%%%%%%%%%%%%%%%%%%%%%%%%%%%%%%%%%%%%%%%%%%%%%%%%%%%%%%%%%%%%%%%%%%%%%%%%%%%%%%%%%%%%%%%%%%%%%%%%%%%%%%%%%%%%%%%%%%%%%%%%%%%%%%%%%%%%%%%%%%%%%%%%%%%%%%%%%%%%%%%%%%%%%%%%%%%%%%%%%%%%%%%%%%%%%%%%%%%%%%%%%%%%%%%%%%%%%%%%%%%%%%%%%%%%%%%%%%%%%%%%%%%%%%%%%%%%%%%%%%%

\subsection{Annealed convergences}

\begin{prop}\label{annealconv}
 For all $\alpha\in(0,1)$ and $\hat \nu\in\mathcal M_{1,V_\alpha}$,
\begin{equation}\label{annealconveq}
\hat \mu\in\mathcal M_{1,V_\alpha}\quad\mbox{and}\quad \lim_{t\to\infty} \| \hat \nu \widehat P_t - \hat \mu \|_{V_\alpha}=0.
\end{equation}
\end{prop}

\begin{proof} 
Let $0<\rho<1$ be and apply Proposition  \ref{lyapunovValpha} to see that, for all $0\leq u\leq 1$,
\begin{equation}\label{majo2}
P_{u}(\theta)V_\alpha \leq (\rho+1) V_\alpha + B_\theta\quad\mbox{and}\quad P_\theta V_\alpha \leq \rho V_\alpha + B_\theta\quad\mathcal W\mbox{-a.s.}
\end{equation}
We get from the latter inequality and (\ref{invconv}) that $\mu_{T\theta}(V_\alpha) \leq  \rho  \mu_\theta(V_\alpha) + B_\theta$ $\mathcal W$-a.s. and, by taking the expectation of the last inequality, we obtain the left hand side of (\ref{annealconveq}). Besides, since the Wiener measure is $(T_{t})$-invariant, we can see that  
\begin{multline}\label{cvd}
\|\hat \nu \widehat P_t - \hat \mu \|_{V_\alpha}
=\| \mathds E_{\scriptscriptstyle \mathcal W} [\hat \nu \widehat P_t(T^{-t}\theta) - \hat \mu_{T_{\{t\}\theta}}]\|_{V_\alpha}\\
\leq  \mathds E_{\scriptscriptstyle \mathcal W}[\|\hat \nu P_t(T^{-[t]}\theta)-\mu_{T_{\{t\}}\theta}\|_{V_\alpha} ].
\end{multline}
Moreover, the relation (\ref{invconv}) and the cocycle property (\ref{cocycle_intro}) allow us to write
\begin{equation*}
P_{t}(T^{-[t]}\theta)= P_{[t]}(T^{-[t]}\theta) P_{\{t\}}(\theta)\quad\mbox{and}\quad \mu_\theta 
P_{\{t\}}(\theta)=\mu_{T_{\{t\}\theta}}.
\end{equation*}
Then similar arguments as for the proofs of (\ref{convinv}) and (\ref{conv0}) hold and we get that
\begin{multline}\label{cvd2}
\lim_{t\to\infty}\|\hat \nu P_t(T^{-[t]}\theta)-\mu_{T_{\{t\}}\theta}\|_{V_\alpha}\\
\leq \lim_{t\to\infty}\|\hat \nu P_{[t]}(T^{-[t]}\theta)-\mu_\theta\|_{V_\alpha}  \|P_{\{t\}}(\theta)\|_{V_{\alpha}} 
=0\quad\mathcal W\mbox{-a.s.}
\end{multline}
Furthermore, by using (\ref{majo2}) and the cocycle property, it is not difficult to see that 
\begin{multline*}
\|\nu P_t(T^{-[t]}\theta)\|_{V_\alpha} \leq  (\rho+1)\Big( \rho \|\nu\|_{V_\alpha} + \sum_{k=0}^\infty \rho^k B(T^{k}\theta)\Big) +  B_\theta\\ 
\mbox{and}\quad 
\|\mu_{T_{\{t\}}\theta}\|_{V_\alpha}  \leq  (\rho+1)\|\mu_\theta\| + B_\theta.
\end{multline*}
Noting that the two previous bounds belong to ${\rm L}^1(\Theta,\mathcal B,\mathcal W)$ (see Proposition \ref{lyapunovValpha}) and are independent of $t\geq 0$, the dominate convergence theorem applies and we deduce from (\ref{cvd2}) and (\ref{cvd}) the right hand side of (\ref{annealconveq}).
\end{proof}

%%%%%%%%%%%%%%%%%%%%%%%%%%%%%%%%%%%%%%%%%%%%%%%%%%%%%%%%%%%%%%%%%%%%%%%%%%%%%%%%%%%%%%%%%%%%%%%%%%%%%%%%%%%%%%%%%%%%%%%%%%%%%%%%%%%%%%%%%%%%%%%%%%%%%%%%%%%%%%%%%%%%%%%%%%%%%%%%%%%%%%%%%%%%%%%%%%%%%%%%%%%%%%%%%%%%%%%%%%%%%%%%%%%%%%%%%%%%%%%%%%%%%%%%%%%%%%%%%%%%%%%%%%%%%%%%%%%%%%%%%%%%%%%%%%%%%%%%%%%%%%%%%%%%%%%%%%%%%%%%%%%%%%%%%%%%%%%%%%%%%%%%%%%%%%%%%%%%%%%%%%%%%%%%%%%%%%%

\section{Proof of Theorem \ref{CG}}
\setcounter{equation}{0}
\label{CGsection}

Recall that under $\mathds P_z(\theta)$  (see Proposition \ref{equivalent_mp}) $\{S_\theta(t,X_t) : t\geq 0 \}$ is a solution of the SDE (\ref{sde}), with $a=1$. Moreover, since $r>0$, we can see by using (\ref{cocycle_coeff}) that 
\begin{multline*}\label{conver3}
\lim_{t\to\infty} S_{\theta}(t,x)=S(x):=\int_{0}^{x} e^{\frac{y^{2}}{2}}\,\dd y,\quad \lim_{t\to\infty} H_{\theta}(t,x)=S^{-1}(x),\\
\lim_{t\to \infty} \sigma_{\theta}(t,x)=S^{\prime}\circ S^{-1}(x)\quad\mbox{and}\quad \lim_{t\to \infty} d_{\theta}(t,x)=0,
\end{multline*}
uniformly on compact sets. Following \cite[Lemma 4.5]{MY} and denoting by $\Gamma$ the standard normal distribution, $\{S_{\theta}(t,X_{t}) : t\geq 0\}$ is asymptotically time-homogeneous and $S_{\ast}\Gamma$-ergodic.  According to the cited Lemma, if in addition $\{S_{\theta}(t,X_{t}) : t \geq 0\}$ is bounded in probability, it converges in distribution towards $S_{\ast}\Gamma$:
\begin{equation*}\label{resultabove}
\big(\forall \ve>0,\;\exists R>0,\;\sup_{t\geq 0}\mathds P_{z}(\theta)(|S_{\theta}(t,X_{t})|\geq R)\leq \ve\big)\Longrightarrow \lim_{t\to\infty} S_{\theta}(t,X_{t})\overset{(\dd)}{=} S_{\ast}\Gamma.	
\end{equation*}
We shall prove that $\{X_t : t\geq 0\}$ is bounded in probability, which shall imply the boundedness in probability of  $\{S_{\theta}(t,X_t) : t\geq 0\}$. By using Proposition \ref{lyapunovUalpha}, we can find $0<\rho<1$, $L>0$, $B:\Theta\longrightarrow [1,\infty)$ and $k,c,p>0$ such that, for all  $0\leq u\leq 1$,
\begin{equation*}\label{majo}
P_u(\theta) U_\alpha \leq L U_\alpha + B_\theta,\;\;\; P_\theta U_\alpha \leq \rho U_\alpha + B_\theta\;\;\;\mbox{and}\;\;\; B_\theta\leq k\, \exp\left[c H_\gamma^p(\theta)\right]\quad \mathcal W\mbox{-a.s.}
\end{equation*}
Then relations (\ref{presque_cocycle_intro}) and the ergodic theorem allow us to write that, for all $t\geq 0$,
\begin{multline*}
\sup_{t\geq 0}   \|P_\theta(t,z,\dd x)\|_{U_\alpha} \leq \sup_{t\geq 0}\; L\Big(\rho^
{[t]} U_\alpha(z) + k\sum_{m=0}^{[t]-1} \rho^{[t]-m} \exp{\left[c e^{-r p m} H_\gamma^p(T^{m}\theta)\right]}\Big) + B_\theta\\
\leq L\Big(\rho U_\alpha(z) + \frac{k}{1-\rho} \exp{\big[\sup_{m\geq 0}\big(c e^{-r p m} H_\gamma^p(T^{m}\theta)\big)\big]}\Big) + B_\theta <\infty\quad\mathcal W\mbox{-a.s.}
\end{multline*}
Thereafter, the Markov inequality implies that 
\begin{equation*}
\sup_{t\geq 0}\mathds P_{z}(\theta)(|X_{t}|\geq R)\leq \frac{\sup_{t\geq 0}  \|P_\theta(t,z,\dd x)\|_{U_\alpha}}{U_{\alpha}(R)}\quad \mathcal W\mbox{-a.s.}
\end{equation*}
Therefore, we get that $\{X_t : t\geq 0\}$ is bounded in probability and since
\begin{equation*}
\lim_{|x|\to\infty}\inf_{t\geq 0} S_{\theta}(t,x)=\infty,	
\end{equation*}
we obtain also the boundedness in probability of $\{S_{\theta}(t,X_t) : t\geq 0\}$. We deduce that \cite[Lemma 4.5]{MY} applies and this completes the proof.\hfill$\Box$\\

\noindent
{\bf Acknoledgements\;} The author is grateful to the Referee for careful reading and valuable comments and remarks which have significantly improve the manuscript.

\bibliographystyle{spmpsci}% mathematics and physical sciences
\bibliography{bibliographie-re.bib}

%%%%%%%%%%%%%%%%%%%%%%%%%%%%%%%%%%%%%%%%%%%%%%%%%%%%%%%%%%%%%%%%%%%%%%%%%%%%

\end{document}